%% file: root.tex
\documentclass[journal,twoside,web]{ieeecolor}
\usepackage{generic}
\usepackage{cite}
\usepackage{amsmath,amssymb,amsfonts}
\usepackage{algorithmic}
\usepackage{graphicx}
\usepackage{algorithm,algorithmic}
\usepackage{hyperref}
\hypersetup{hidelinks=true}
\usepackage{textcomp}
\def\BibTeX{{\rm B\kern-.05em{\sc i\kern-.025em b}\kern-.08em
    T\kern-.1667em\lower.7ex\hbox{E}\kern-.125emX}}
\usepackage[utf8]{inputenc} 
\usepackage[english]{babel} 
\usepackage{amsmath,amssymb,mathtools} 
\usepackage{graphicx} 
\graphicspath{{images/}} 
\usepackage{float}

\usepackage{Macros}
\newtheorem{lemma}{Lemma}[section]

\newtheorem{proposition}[lemma]{Proposition}
\newtheorem{theorem}[lemma]{Theorem}

\newtheorem{corollary}[lemma]{Corollary}
\newtheorem{definition}[lemma]{Definition}
\newtheorem{example}[lemma]{Example}

\newtheorem{remark}[lemma]{Remark}

\renewcommand\theenumi{\arabic{enumi})}

\allowdisplaybreaks[4]  

\newcommand{\Uc}{\ensuremath{\mathcal{U}}}

\usepackage{tikz}
\usepackage{pgfplots}
\pgfplotsset{compat=1.18}
\usetikzlibrary{positioning}
\usetikzlibrary{cd}
\usetikzlibrary{shapes,fit} 

\usetikzlibrary{decorations.markings}
\tikzset{degil/.style={
		decoration={markings,
			mark= at position 0.5 with {
				\node[transform shape] (tempnode) {$\backslash$};
			}
		},
		postaction={decorate}
	}
}
\definecolor{uniwueblue}{HTML}{1F5394}
\pgfplotsset{soldot/.style={color=uniwueblue,only marks,mark=*, mark options={scale=.3mm}}}

\renewcommand{\pb}[1]{{\color{black} #1}} 
\renewcommand{\amc}[1]{{\color{black} #1}} 

\begin{document}
\title{Characterization of input-to-output stability for infinite-dimensional systems}

\author{Patrick Bachmann$^{1,2}$,~\IEEEmembership{Student Member,~IEEE,} Sergey Dashkovskiy$^{2}$,~\IEEEmembership{Senior Member,~IEEE,} and \\
Andrii Mironchenko$^{1}$,~\IEEEmembership{Senior Member,~IEEE} 
\thanks{*A. Mironchenko has been supported by the Heisenberg grant (MI 1886/3-1) of the German Research Foundation (DFG). \amc{P. Bachmann is supported by DFG (grant MI 1886/5-1).}} 
\thanks{$^{1}$\pb{P Bachmann and }A. Mironchenko are with the Department of Mathematics, University of Bayreuth, Germany {\tt\small andrii.mironchenko@uni-bayreuth.de}}%
\thanks{$^{2}$P. Bachmann and S. Dashkovskiy are with the Institute of Mathematics, University of W{\"u}rzburg, Germany {\tt\small patrick.bachmann@uni-wuerzburg.de}, {\tt\small sergey.dashkovskiy@uni-wuerzburg.de}}%
}

\maketitle
\thispagestyle{empty}
\pagestyle{empty}

\begin{abstract}
We prove a superposition theorem for input-to-output stability (IOS) of a broad class of nonlinear infinite-dimensional systems with outputs including both continuous-time and discrete-time systems. 
It contains, as a special case, the superposition theorem for input-to-state stability (ISS) of infinite-dimensional systems 
and the IOS superposition theorem for systems of ordinary differential equations \pb{known from the literature.} 

To achieve this result, we introduce and examine several novel stability and attractivity concepts for infinite-dimensional systems with outputs: We prove criteria for the uniform limit property for systems with outputs, several of which are new already for systems with full-state output, we provide superposition theorems for systems which satisfy both the output Lagrange stability (OL) and IOS, give a sufficient condition for OL and characterize ISS in terms of IOS and input/output-to-state stability.
Finally, by means of counterexamples, we illustrate the challenges appearing on the way of extension of the superposition theorems from \pb{the literature} 
to infinite-dimensional systems with outputs.

%
\end{abstract}


\section{Introduction}



Input-to-state stability (ISS) was first introduced for systems of ordinary differential equations (ODEs) \cite{Son89}, and then developed for other classes of finite-dimensional control systems such as switched \cite{MancillaAguilar2001}, hybrid \cite{Cai2005}, and impulsive systems \cite{Dashkovskiy2013}.
More recently, the ISS theory was extended to infinite-dimensional systems, including time-delay systems \cite{Pepe2006,CKPW23}, partial differential equations (PDEs) \cite{Karafyllis2019} and general evolution equations in Banach spaces \cite{Sch20, Mironchenko2020}. For more details, we refer to the survey \cite{Mironchenko2020}. ISS of infinite-dimensional discrete-time systems was treated, e.g., in \cite{Bachmann2022a, Dashkovskiy2023a}.

Yet, the developments above are confined to systems for which the output equals the state.
A notion extending ISS to systems with outputs is given by input-to-output stability (IOS) introduced for ODE systems in \pb{\cite{JTP94}} (though it was considered by Sontag earlier in the input/output formalism in \cite{Son89}).  
IOS combines the uniform global asymptotic stability of the output dynamics with its robustness w.r.t. external inputs. 
If the output equals to the state, IOS coincides with ISS. We will refer to this case by the term \emph{full-state output}.
Other choices for the output function can be: partial state output, e.g. due to sensor measurements, tracking error, observer error, drifting error from a targeted set, etc. In this context, IOS represents robust stability of a control system with respect to the given errors.
\pb{In \cite{Vorotnikov2005}, several stability notions for unperturbed systems  with outputs are discussed, e.g., uniform global asymptotic $y$-stability, which can be interpreted as IOS with zero input and partial uniform stability, which is a relaxation of ouput Lagrange stability (OL). Stability with respect to two measures \cite{TeP00} 
generalizes 
IOS such that general continuous, positive semidefinite functions on the state and output space, respectively, are considered rather than the norms on both spaces.}

IOS is paramount in numerous applications including multi-agent systems \cite{Liu2013}, coverage controllers \cite{Kennedy2024}, and neural networks \cite{Mei2022}.

{\bf IOS theory for ODEs.} Already for ODE systems, the IOS theory is significantly more involved than the ISS theory for systems with full-state output. Lyapunov criteria of IOS have been shown in \cite{Sontag2000} 
 based on some earlier developments in \cite{Sontag1999}. 
 However, the results in \cite{Sontag2000} are obtained
 under the assumption that the system is uniformly globally stable -- which was not required in converse Lyapunov results for ISS. 

 Seminal ISS superposition results for ODE systems from \cite{SoW96} have been extended to the IOS case only under the assumption of 
 OL, which was not needed in the ISS case. In particular, in \cite{ISW01}, it is shown that an ODE system is IOS if it satisfies both OL and the output-limit property (OLIM) (also see \cite{Angeli2004a} for several results omitted in \cite{ISW01}, e.g., Lem. 2.2 and Cor. 2.3).







Trajectory-based small-gain theorems for interconnections of two IOS systems have been obtained in \cite{JTP94} and generalized to interconnections of $n \in \N$ IOS systems in \cite{Jiang2008}. Lyapunov-based small-gain theorems for couplings of $n$ interconnected IOS systems have been reported in \cite[Sec. 3.3.4]{Rueffer2007}.



\textbf{Infinite-dimensional IOS theory.}
In time-delay context, IOS serves for controller design in networked systems, which is applied to teleoperating systems, though in this case only weaker than IOS properties for the control system are obtained \cite{Polushin2013}. The work \cite{Gruene2022} develops finite-dimensional observer-based controllers for a linear reaction-diffusion system. In \cite{KaJ11, Bao2018}, small-gain theorems are presented which are tailored for the so-called maximum formulation of the IOS property. For time-delay systems, Lyapunov characterizations of IOS were developed (cf. \cite{CKPW23}). 

Nevertheless, despite its practical relevance, infinite-dimensional IOS theory remains largely unexplored \cite{Mironchenko2020}.


\textbf{Challenges.} 
In addition to obstructions encountered in characterizing IOS for ODE systems \cite{ISW01} (such as the need for analysis of OL), infinite-dimensionality of control systems leads to several additional challenges.
In \cite[Ex. 1]{MiW18b}, it is shown directly that OLIM and OL are insufficient to imply IOS for the case of linear full-state output infinite-dimensional systems due to the lack of uniformity of the limit property.
Similarly, the IOS characterization for ODE systems in terms of OL and the output-asymptotic gain property (OAG) cannot be extended to the infinite-dimensional setting in the same formulation, even in the ISS case with full-state output, because trajectory-wise asymptotic stability does not imply uniform asymptotic stability as argued in \cite[Lem. 9]{MiW18b}. 

A different problem arises due to the fact that nonlinear forward complete infinite-dimensional systems do not necessarily have bounded reachability sets, in contrast to nonlinear ODE systems \cite{MiW18b}.
As we discuss in Section \ref{sec:counterexamples}, one of the consequences of this problem is the breakdown of the equivalence between several types of uniform asymptotic gain properties, in contrary to the finite-dimensional case.
In view of this, the investigation of the IOS of \emph{infinite-dimensional nonlinear systems} becomes challenging.



\textbf{Contribution.} 
Motivated by the infinite-dimensional ISS superposition theorem \cite{MiW18b}, \emph{we characterize the IOS property for infinite-dimensional continuous and discrete-time systems in terms of weaker properties}, such as the output-uniform asymptotic gain property (OUAG), output-uniform local stability (OULS), output continuity at the equilibrium point (OCEP) and other notions. To support this, the relation between OUAG and its variations output-global UAG (OGUAG) and output-complete AG (OCAG) as well as input-to-output practical stability (IOpS) is investigated.

Furthermore, we consider the influence of OL on the IOS property by establishing a superposition theorem for systems which are OL and IOS. We provide a superposition theorem for OL. 
We compare our results applied to the finite-dimensional case with the ones in \cite{ISW01} and show the equivalence of OUAG and OGUAG as well as of several notions of OLIM for ODE systems.
Moreover, our results extend the ISS superposition theorem for distributed parameter systems shown in \cite[Thm. 5]{MiW18b}. We characterize the ISS property in terms of IOS and input/output-to-state stability (IOSS), thereby providing a partial extension of \cite[Prop. 3.1]{JTP94} for infinite-dimensional systems.
We point out differences between the ISS case and general IOS case by several (counter)examples.
\pb{The main results are depicted in Figure \ref{fig:mainResult}.}

\textbf{Significance.} 
IOS superposition theorems are a meta-tool that helps to prove other important theoretical results including Lyapunov theory and small-gain theorems.
Recently, in \cite{MWC24b}, ISS characterizations have been used to prove  Lyapunov-Krasovskii theorems with pointwise dissipation for ISS of nonlinear time-delay systems. Our IOS characterizations can be a basis that will help to extend those results to IOS Lyapunov-Krasovskii theorems. 





These IOS characterizations can be applied to extend a small-gain theorem to infinite networks of infinite-dimensional IOS subsystems. For ISS, \cite{Mironchenko2021c} provides such a general small-gain theorem based on the ISS superposition theorems \cite{MiW18b}. 
From a more remote perspective, by exploiting information about time delays in the interconnection structure of the network, one could formulate 
stronger IOS small-gain theorems tailored for time-delay systems, which will go far beyond existing results even in the ISS case.

Preliminary results of this work were presented in \cite{BDM24a}.

\textbf{Organization.} 
The rest of the paper unfolds as follows. In Section \ref{sec:preliminaries}, we give the preliminaries and introduce the main concepts to be investigated in this article. In Section~\ref{sec:superpositionThms}, we provide our main results such as IOS Superposition Theorem~\ref{thm:IOSequivalences}, characterizations of OCAG (Proposition \ref{prop:OCAGequivalences}) and equivalences for IOS $\land$ OL (Proposition \ref{prop:OLIOSequivalences}). Furthermore, we give a sufficient condition for OL in Lemma \ref{lem:OOULIMandLocalOLandOBORStoGlobalOL}. We apply our results to finite-dimensional ODE systems and compare them with the existing literature on finite-dimensional IOS theory in Section \ref{sec:finiteDimIOSTheory}. Precisely, Proposition \ref{prop:OLIMforFiniteDimensions} shows that OLIM and the related notions output-uniform LIM and output-global uniform LIM are equivalent for ODE systems and in Proposition \ref{prop:IOSsuperpositionsFiniteDim}, we characterize IOS and other properties. Section \ref{sec:counterexamples} is dedicated to counterexamples in order to demonstrate the relation between certain stability notions and emphasize the difficulties in the infinite-dimensional IOS theory. We conclude with a summary and outlook in Section~\ref{sec:conclusion}.

\textbf{Notation.} 
We denote the nonnegative integers by $\N_0$, the natural numbers by $\N$, the real numbers by $\R$, the nonnegative real numbers by $\pb{\R_+}$ and the \pb{open} balls of radius $r$ around zero in Banach spaces $X$, $U$ and $\U$, respectively, by $B_r$, $B_{r,U}$ and $B_{r,\U}$ as well as the \pb{open} ball of radius $r$ around a set $K$ in a Banach space $X$ by 
\[
B_r(K) \coloneq \braces{x \in X\,\middle|\, \exists x_0 \in K \colon \norm{x - x_0}_X < r}.
\]
For a subset $\Omega$ of a Banach space, we denote its set complement by $\Omega^C$ and its closure by $\overline{\Omega}$. We define the standard classes of comparison functions (cf. \cite[p. xvi]{Mir23}) by
\begin{align*}
    \KK &\coloneq \{\gamma: \R_+ \to \R_+ \,|\, \gamma(0) = 0,\ \gamma \text{ is continuous} \\
    &\hspace{3cm}\text{and strictly increasing}\}, \\
    \KK_\infty &\coloneq \{\gamma\in \KK \,|\, \gamma \text{ is unbounded}\}, \\
    \LL &\coloneq \{\gamma: \R_+ \to \R_+ \,|\, \gamma \text{ is continuous and decreasing}  \\
    &\hspace{3cm}\text{with }\lim_{t \to \infty}\gamma(t) = 0\}, \\
    \KK\LL &\coloneq \{\beta \in \R_+ \times \R_+ \to \R_+ \,|\, \beta(\ph,t) \in \KK, \ \forall t \geq 0, \\
    &\hspace{3cm}\beta(r, \ph) \in \LL,\ \forall r > 0\}.
\end{align*}

Let \pb{$I \in \{\N_0,\R_+\}$ denote the positive time set}, $Z$ be a Banach space and $f\colon I \to Z$. We define for $a,b \in I\colon a \leq b$ \amc{the order interval $[a,b]:=\{z \in I: a \leq z\leq b\}$, and} a restriction $f|_{[a,b]}\colon I \to Z$ by
\begin{align*}
    f|_{[a,b]}(s) \coloneq 
    \begin{cases}
        f(s), &\text{if } s \in [a,b], \\
        0, & \text{else.} 
    \end{cases}
\end{align*}
By $\LL^\infty(I, Z)$, we denote the Lebesgue space of strongly measurable functions $f \colon I \to Z$ with norm $\norm{f}_\infty \coloneq \esssup_{t \in I}\norm{f(t)}_Z$.













\section{Preliminaries}\label{sec:preliminaries}

\begin{definition}\label{def:controlSystem}
    Consider a quadruple $\Sigma = (I, X, \U, \phi)$ consisting of
    \begin{enumerate}
        \item a \emph{time set} $I \in \{\N_0, \R_+\}$.
        \item a normed vector space $(X, \norm{\ph}_X)$, called the \emph{state space}.
        \item a vector space $U$ of input values and a normed vector space of inputs $(\U, \norm{\ph}_\U)$, where $\U$ is a linear subspace of $\{u\,|\,u: I \to U\}$. We assume that the following invariance axioms hold:
        \begin{itemize}
            \item \emph{axiom of shift invariance}: for all $u \in \U$ and all $\tau \in I$, the time-shifted function $u(\ph + \tau)$ belongs to $\U$ with $\norm{u}_\U \geq \norm{u(\ph + \tau)}_\U$.
            \item \emph{axiom of restriction invariance}: for each $u \in \U$ and for all $t_2 \geq t_1 \geq 0$ the restriction of $u$ to time interval $[t_1,t_2]$ given by $u|_{[t_1,t_2]}$ belongs to $\U$ and $\norm{u|_{[t_1,t_2]}}_{\U} \leq \norm{u}_{\U}$.
        \end{itemize}
        \item a map $\phi\colon D_\phi \to X$, $D_\phi \subset I \times X \times \U$, called \emph{transition map}, so that for all $(x,u) \in  X \times \U$ it holds that $D_\phi \cap (I \times \{(x,u)\}) = [0,t_m) \times \{(x,u)\}$, for a certain $t_m = t_m(x,u) \in (0, + \infty]$. The corresponding interval $[0,t_m)$ is called the \emph{maximal domain of definition} of the mapping $t \mapsto \phi(t, x, u)$, which we call a \emph{trajectory} of the system.
    \end{enumerate}



    The quadruple $\Sigma$ is called a \emph{(control) system} if it satisfies the following axioms.
    \begin{enumerate}
        \renewcommand\theenumi{($\Sigma$\arabic{enumi})}
        \item\label{cond:identityProperty} \emph{Identity property}: For all $(x,u) \in X \times \U$, it holds that $\phi(0, x, u) = x$.
        \item \emph{Causality}: For all $(t,x,u) \in D_\phi$ and all $\widetilde u \in \U$ such that $u(s)= \widetilde u(s)$ for all $s \in [0,t]$, it holds that $[0,t] \times \{(x, \widetilde u)\} \subset D_\phi$ and $\phi(t,x,u) = \phi(t,x,\widetilde u)$.
        \item\label{cond:cocycleProperty} \emph{Cocycle property}: For all $x \in X$, $u \in \U$ and $t,s \geq 0$ so that $[0,t + s] \times \{(x,u)\} \subset D_\phi$, we have ${\phi(t + s, x, u) = \phi(s, \phi(t, x, u), u(t + \ph))}$.
    \end{enumerate}
\end{definition}

\begin{remark}
    The axiom of restriction invariance is non-trivial: First, the restriction of a function is, in general, not an element of the same function space, e.g., for the space of continuous functions. Second, $\norm{u|_{[t_1,t_2]}}_{\U} \leq \norm{u}_{\U}$ is not satisfied in general even if $u|_{[t_1,t_2]} \in \U$. To illustrate this, we first define the space of Radon measures $\mathcal M$ with norm 
    \begin{align*}
        \norm{f}_{\mathcal M} \coloneq \sup\!\braces{\int_0^\infty \hspace{-1em} f(t) \varphi(t) \, \diff t\,\middle|\, \varphi \in C^1_c(\R_+,\R),\, \norm{\varphi}_1 \leq 1}\!,
    \end{align*}
    where $C^1_c(\R_+,\R)$ denotes the space of continuously differentiable functions from $\R_+$ to $\R$ with compact support.
    
    Next, we consider $\U$ as the space of functions with bounded variation with norm $\norm{u}_{\U} \coloneq \norm{u}_\infty + \norm{\frac{\partial}{\partial t}u}_{\mathcal M}$.
    Then, for $u \in \U$, $u \equiv 1$, it follows that
    $\norm{u|_{[0,1]}}_{\U} = 2 > 1 = \norm{u}_{\U}$.
    \hspace*{\fill}~\QED 
\end{remark}
\begin{remark}
    In \cite{MiW18b}, the additional axiom of \emph{continuity} of the trajectories is introduced. This property is not a requirement for the present paper and in \cite{MiW18b}, it was only used in Proposition 10, but the proof can be adapted to include non-continuous trajectories by the variation we propose in Lemma \ref{lem:OOULIMandLocalOLandOBORStoGlobalOL}. The same holds for the \emph{axiom of concatenation} in Definition \ref{def:controlSystem}. Instead, we introduce the axiom of restriction invariance, which is necessary for our results in Section \ref{sec:CharacterizationOfISS}.
    \hspace*{\fill}~\QED 
\end{remark}
\begin{remark}
    Note that for many of the following results, the cocycle property \ref{cond:cocycleProperty} is not a necessary precondition. E.g., for Theorem \ref{thm:IOSequivalences} and Proposition \ref{prop:OLIOSequivalences}, for the implications  1)$\implies$2)$\implies$3) this condition is not required. Systems that are not satisfying the cocycle property but only the \emph{weak semi-group property} are studied in \cite{KaJ11,KaJ11b}.
    
    However, for a uniform system definition and better readability, we restrict the setting to include the cocycle property.
    \hspace*{\fill}~\QED  
\end{remark}
\begin{definition}
    A (time-invariant) \emph{control system with outputs} $\Sigma \coloneq (I, X, \U, \phi, Y,h)$ is given by an abstract control system $(I, X, \U, \phi)$ together with
    \begin{enumerate}
        \item a normed vector space $(Y, \norm{\ph}_Y)$ called the \emph{output-value space} or \emph{measurement-value space}; and
        \item a map $h\colon X \times U \to Y$, called the \emph{output} (or: \emph{measurement}) \emph{map}.
        %
        %
    \end{enumerate}
    We also denote 
    $y(\ph, x, u) \coloneq h(\phi(\ph, x, u), u(\ph))$ for all $(x, u) \in X \times \U$.
\end{definition}
\pb{This broad class of control systems with outputs includes 
ODEs, impulsive and switched systems and time delay systems. Moreover, semi-linear evolution PDEs generating a $C_0$-semigroup, certain subclasses of well-posed linear systems \cite[Sec. 5]{TuW14} and boundary control systems with sufficient regularity \cite[Sec. 4.1]{Mironchenko2020}, \cite{Sch20} are covered by this framework.}

The following definition is taken from \cite{MiW18b}.
\begin{definition}\label{def:forwardComplete}
    We call a control system $(I,X,\U,\phi)$ \emph{forward complete (FC)}, if for each $x \in X$, $u \in \U$ and $t \in I$ the value $\phi(t,x,u) \in X$ is well-defined.
\end{definition}

In the following, we always consider a forward complete control system with outputs $\Sigma = (I, X, \U, \phi, Y,h)$.

\begin{definition}\label{def:OCEP}
    We call $\Sigma$ \emph{output continuous at the equilibrium point (OCEP)} if for every $\tau \in I$ and every $\eps > 0$ there exists $\delta = \delta(\eps, \tau) > 0$ such that 
    \begin{align*} 
		t \in I\colon t \leq \tau, \, \norm{x}_X \leq \delta, \, \norm{u}_{\U} \leq \delta \implies \norm{y(t,x,u)}_Y \leq \eps.
	\end{align*}
\end{definition}

\begin{definition}\label{def:BORS}
    $\Sigma$ is said to have \emph{bounded output reachability sets (BORS)} if for all $C > 0$ and $\tau \in I$ it holds that
	\begin{align*} 
		\sup_{\norm{x}_X < C,\, \norm{u}_\U < C, \ t < \tau}\norm{y(t,x,u)}_Y < \infty. 
	\end{align*}
\end{definition}

\begin{definition}\label{def:hBoundedOnBoundedSets}
    The output map $h$ is called \emph{bounded on bounded sets} if for all $C > 0$, there exists $D = D(C)$ such that for all $x \in B_C$, $u \in B_{C,\pb{\U}}$, the bound $\norm{h(x,u\pb{(t)})}_Y \leq D$ holds \pb{for all $t \in I$}.
\end{definition}

Boundedness on bounded sets can be equivalently characterized as follows \cite[Lem. 3]{MiW18b}.
\begin{lemma}
    The output map $h$ is bounded on bounded sets if and only if there exist $\sigma_1, \gamma_1 \in \KK$ and $c \geq 0$ such that for all $x\in X$\pb{, $u \in \U$ and $t \in I$} we have 
    \begin{align}
        \norm{h(x,u\pb{(t)})}_Y \leq \sigma_1(\norm{x}_X) + \gamma_1(\norm{u}_{\U}) + c. \label{ineq:hBoundedOnBoundedSets}
    \end{align} 
\end{lemma}

\begin{proof}
    Let $h$ be bounded on bounded sets. Then there exists $\mu \colon \R_+ \times \R_+ \to \R_+$, which is continuous and component-wise increasing, such that \pb{for all $x \in X$, $u \in \U$, $t \in I$:}
    \begin{align*}
        \norm{h(x,u\pb{(t)})}_Y \leq \mu(\norm x_X,\norm u_{\U}). 
    \end{align*}
    Then, by the choice $\sigma_1(r) = \gamma_1(r) \coloneq \mu(r,r) - \mu(0,0)$, and $c =  \mu(0,0)$, it follows that
    \begin{align*}
        \norm{h(x,u\pb{(t)})}_Y &\leq \sigma_1\!\paren{\max\!\braces{\norm x_X,\norm u_{\U}}} + c \\
        &\leq \sigma_1\!\paren{\norm x_X} + \gamma_1\!\paren{\norm u_{\U}}+ c,
    \end{align*}
    as desired. The converse statement is clear.
\end{proof}
\pb{
\begin{remark}
    For inputs defined in almost everywhere-sense, boundedness on bounded sets of $h$, \eqref{ineq:hBoundedOnBoundedSets} is in general not well-defined as $u$ cannot be evaluated at $t$. 
    
    One of the ways to remedy this is to impose a stronger assumption that for all $w \in U$ we have that
    \begin{align*}
        \norm{h(x,w)}_Y \leq \sigma_1(\norm{x}_X) + c.
    \end{align*}
    
    Also see \cite[Sec. 4.1]{Jacob2020} and especially Remark 4.6 for an analogous phenomenon.
    
    In case $h$ does not satisfy this stronger condition, it seems legitimate to weaken boundedness on bounded sets of $h$ in the sense that \eqref{ineq:hBoundedOnBoundedSets} 
    holds only for almost every $t \in I$. In this case, all the following definitions need to be redefined canonically to hold for almost every $t \in I$. Formally, it is only necessary to adapt the time domain such that ``$\forall t \in I$'' is replaced by ``$\exists$ a null set $\mathcal N\colon \forall t \in I\setminus \mathcal N$''. Note that in this context, the weak attractivity notions Definitions \ref{def:OLIM}--\ref{def:OULIM} play a special role as for their adaptation 
    ``$\exists t \in I$''  must be replaced by ``$\forall$ null sets $\mathcal N\colon \exists t \in I\setminus \mathcal N$''. 
    
    Rigorously applied to all introduced stability notions in Table \ref{tab:Abbreviations}, these alternative definitions lead to the same results throughout the paper with only minor adaptations of the respective proofs. However, for simplicity of notation, we keep the pointwise notions.
    \hspace*{\fill}~\QED
\end{remark}
}






\begin{definition}\label{def:hKbounded}
    We call the output map $h$ \emph{$\KK$-bounded} if there exist $\sigma_1, \gamma_1 \in \KK$ such that for all $x\in X$ and all $u\in \U$ we have \pb{for all $t \in I$:}
    \begin{align}\label{ineq:boundedOutputMap}
        \norm{h(x,u\pb{(t)})}_Y \leq \sigma_1(\norm{x}_X) + \gamma_1(\norm{u}_{\U}).
    \end{align} 
    
\end{definition}

Let us define the main concept of this paper.
\begin{definition}\label{def:IOS}
    $\Sigma$ is called \emph{input-to-output stable (IOS)}, if there exist $\beta \in \KK\LL$ and $\gamma \in \KK_\infty$ such that $\forall x \in X$, $\forall u \in \U$ the following holds:
    \begin{align}\label{ineq:IOS}
        \norm{y(t,x,u)}_Y \leq \beta\!\paren{\norm{x}_X,t} + \gamma\!\paren{\norm{u}_\U}\!, \qquad t \in I.
    \end{align}   
    
\end{definition}


Based on the notion of input/output stability in \cite{Son89}, the concept of IOS was introduced \pb{for ODEs} in \cite{JTP94}. IOS generalizes the following definition of input-to-state stability \cite{Son89} to systems with outputs as described in Remark \ref{rem:ISSvsIOS}.
\begin{definition}\label{def:ISS}
    $\Sigma$ is called \emph{input-to-state stable (ISS)}, if there exist $\beta \in \KK\LL$ and $\gamma \in \KK_\infty$ such that $\forall x \in X$, $\forall u \in \U$ the following holds:
    \begin{align}\label{ineq:ISS}
        \norm{\phi(t,x,u)}_X \leq \beta\!\paren{\norm{x}_X,t} + \gamma\!\paren{\norm{u}_\U}\!, \qquad t \in I.
    \end{align}
\end{definition}
\begin{remark}\label{rem:ISSvsIOS}
    A special case of output systems is given for $Y = X$, $h(x,u) \equiv x$ and $y(t,x,u) = \phi(t,x,u)$ for all $t\in I$, $x \in X$ and $u\in \U$. We will refer to this kind of systems by \emph{systems with full-state output}.
    For such systems, IOS reduces to ISS.
    \hspace*{\fill}~\QED
\end{remark}

We also introduce the following property which is weaker than IOS and was introduced in \cite{JTP94}.
\begin{definition}\label{def:IOpS}
    $\Sigma$ is called \emph{input-to-output practically stable (IOpS)}, if there exist $\beta \in \KK\LL$, $\gamma \in \KK_\infty$ and $c \geq 0$ such that $\forall x \in X$, $\forall u \in \U$ the following holds:
    \begin{align*}
        \norm{y(t,x,u)}_Y \leq \beta\!\paren{\norm{x}_X,t} + \gamma\!\paren{\norm{u}_\U} + c, \qquad t \in I.
    \end{align*}   
    \amc{Here, $c$ is called the \emph{residual constant}.}
\end{definition}

\subsection{Stability properties}

In this section, we introduce several stability properties needed for the characterization of IOS.

\begin{definition}[\!\!\cite{Sontag1999}]\label{def:OL}
	We call $\Sigma$ \emph{output Lagrange stable (OL)} if there exist $\sigma, \gamma \in \KK_\infty$ such that for all $x \in X$ and $u \in \U$, it holds that
	\begin{align}\label{ineq:defOL}
		\norm{y(t,x,u)}_Y \leq \sigma\!\paren{\norm{y(0,x,u)}_Y} + \gamma\!\paren{\norm{u}_{\U}}\!, \quad t \in I.
	\end{align}
    We call $\Sigma$ \emph{locally output Lagrange stable (locally OL)} if there exist $\sigma, \gamma \in \KK_\infty$ and $r > 0$ such that for all $x \in B_r$ and $u \in B_{r,\U}$, 
    \eqref{ineq:defOL} holds.
\end{definition}

The following notions generalize the classical concepts of uniform local/global stability (cf. \cite{MiW18b}) to systems with outputs. 

\begin{definition}\label{def:OULS}
	We call system $\Sigma$ 
    \begin{enumerate}
        \item \emph{output-uniformly locally stable (OULS)} if there exist ${r > 0}$ and $\sigma, \gamma \in \KK_\infty$ such that for all $x \in B_r$ and $u \in B_{r,\U}$, it holds that
    	\begin{align}\label{ineq:OULS}
    		\norm{y(t,x,u)}_Y \leq \sigma\!\paren{\norm{x}_X} + \gamma\!\paren{\norm{u}_{\U}}\!, \qquad t \in I;
    	\end{align}
        \item \emph{output-uniformly globally stable (OUGS)} if there exist $\sigma, \gamma \in \KK_\infty$ such that for all $x \in X$, $u \in \U$, 
        \eqref{ineq:OULS} holds.
        \item \emph{output-uniformly globally bounded (OUGB)} if there exist $\sigma, \gamma \in \KK_\infty$, $c > 0$ such that for all $x \in X$ and all $u \in \U$, it holds that
    	\begin{align*} 
    		\norm{y(t,x,u)}_Y \leq \sigma\!\paren{\norm{x}_X} + \gamma\!\paren{\norm{u}_{\U}} + c, \qquad t \in I.
    	\end{align*}
    \end{enumerate}
\end{definition}


An equivalent characterization of local OL and OULS in $\eps$-$\delta$-notation is given by the following
\begin{lemma}\label{lem:redefineOLOULS}
    Consider a control system with outputs $\Sigma = (I, X, \U, \phi, Y,h)$.
    \begin{enumerate}
        \item $\Sigma$ is locally OL if and only if for all $\eps > 0$, there exists $\delta > 0$ such that
        \begin{align*} 
    		\norm{y(0,x,u)}_Y &\leq \delta, \quad \norm{u}_{\U} \leq \delta\\ 
            &\qquad\implies \norm{y(t,x,u)}_Y \leq \eps,\quad t \in I.
    	\end{align*}
        \item System $\Sigma$ is OULS if and only if for all $\eps > 0$, there exists $\delta > 0$ such that
        \begin{align*} 
    		\norm{x}_X \leq \delta, \;\; \norm{u}_{\U} \leq \delta \implies \norm{y(t,x,u)}_Y \leq \eps,\;\;  t \in I.
    	\end{align*}
    \end{enumerate} 
\end{lemma}

\begin{proof}
    The proof is analogous to the proof of \cite[Lem.~2]{MiW18b}.
\end{proof}

The notions of OULS and local OL (OUGS and OL) coincide for systems with full-state output. For systems with full-state output, OULS and local OL become uniform local stability (ULS), OUGS and OL are the same as uniform global stability (UGS).  OUGB is uniform global boundedness (UGB) as defined in \cite{MiW18b}. Similarly, many of the other notions are derived from a concept for systems with full-state output which has the same name except the word \emph{output} in the beginning.


\subsection{Attractivity properties}
Following \cite{ISW01}, we define several attractivity-like properties for systems with inputs and outputs, and use them to characterize IOS.
\begin{definition}\label{def:OGUAG}
    $\Sigma$ has the
    \begin{enumerate}
        \item \emph{output-global uniform asymptotic gain property (OGUAG)} if there exists
    	$\gamma \in \KK_\infty$ such that for every $\eps > 0$, and every $r>0$,  there exists $\tau = \tau(\eps, r) \in I$ such that
    	\begin{align*} 
        	\norm{y(t,x,u)}_Y &\leq \eps + \gamma\!\paren{\norm{u}_{\U}}\!,\quad x \in B_r, \, u\in \U, \, t \geq \tau;
    	\end{align*}
        \item \emph{output-uniform asymptotic gain property (OUAG)} if there exists
    	$\gamma \in \KK_\infty$ such that for every $\eps,r,s > 0$, there exists $\tau = \tau(\eps, r,s) \in I$ such that
    	\begin{align*} 
        	\norm{y(t,x,u)}_Y &\leq \eps + \gamma\!\paren{\norm{u}_{\U}}\!, \\
            &x \in B_r, \, u\in B_{s,\U}, \, t \geq \tau;
    	\end{align*}
        
        \item \emph{output-asymptotic gain property (OAG)} if there exists
    	$\gamma \in \KK_\infty$ such that for every $\eps > 0$, $x \in X$ and $u\in \U$, there exists $\tau = \tau(\eps, x, u) \in I$ such that
    	\begin{align*} 
        	\norm{y(t,x,u)}_Y &\leq \eps + \gamma\!\paren{\norm{u}_{\U}}\!, \quad t \geq \tau.
    	\end{align*}
    \end{enumerate}
\end{definition}
A system is OGUAG, OUAG and OAG, respectively, if all outputs converge to the ball with radius $\gamma(\norm{u}_\infty)$. The difference between them is that for OUAG, the convergence rate depends on the norm of the input and the norm of the state of the system, and for OGUAG it depends on the norm of the state, but not on the applied input. For OAG, the convergence rate is individual to each state and input.

As stated in \cite[Thm. 1]{SoW96}, OAG and OGUAG are not equivalent for finite-dimensional systems even in the case of full-state output. Following the lines of proof of \cite[Prop. I.1]{SoW96}, even the stronger \amc{negative} result OAG$\nimplies$OUAG is true. Therefore, an important question is the relation between OUAG and OGUAG. In Proposition \ref{prop:OCAGequivalences}, we will show that for systems satisfying BORS, the properties OGUAG and OUAG are equivalent notions. Opposed to that, we demonstrate in Example \ref{ex:OUAGnottoOGUAG} that the notions are in general not equivalent if BORS is not satisfied.

We proceed with an equivalent characterization of OUAG.
\begin{lemma}\label{lem:alternativeOUAG}
    $\Sigma$ is OUAG if and only if there exists
	$\gamma \in \KK_\infty$ so that for all $\eps, r, s > 0$, there is $\tau = \tau(\eps,r,s) \in I$ with
	\begin{align}\label{ineq:alternativeOUAG}
		\norm{y(t,x,u)}_Y &\leq \eps + \gamma\!\paren{s}\!, \quad  
        x \in B_r, \, u\in B_{s,\U}, \, t \geq \tau.
	\end{align}
\end{lemma}
The difference to the definition of OUAG is that \eqref{ineq:alternativeOUAG} is an upper bound in terms of $s$ instead of $u\in B_{s,\U}$.

\begin{proof}
    It is clear that OUAG implies \eqref{ineq:alternativeOUAG} as $\norm u_{\U} < s$ and $\gamma$ is strictly increasing.

    We show that the converse holds. Let $\gamma$ be as in \eqref{ineq:alternativeOUAG}. We fix $\eps, r,s > 0$. Then, for every $k \in \N$, there exists $\tau_k \coloneq \tau\!\paren{\frac{\eps}{2},r,e^{-k + 1}s}$ such that for all $x \in B_r$, all $u$ for which $\norm u_{\U} \in [e^{-k}s,e^{-k + 1}s)$ and all $t \in I\colon  t \geq \tau_k$, it holds that 
    \begin{align}\label{ineq:alternativeOUAGforbiginputs}
		\norm{y(t,x,u)}_Y 
        \leq \tfrac{\eps}{2} + \gamma\!\paren{e^{-k+1}s}
        \leq \tfrac{\eps}{2} + \gamma\!\paren{e \cdot \norm u_{\U}}\!.
	\end{align}
    
    Let $k^* \in\N$ be such that $\gamma\!\paren{e^{-k^*+1}s} \leq \frac{\eps}{2}$.
    
    Then, for all $x \in B_r$, $u \in B_{e^{-k^* + 1}s,\U}$ and $t \geq \tau_{k^*}$, it holds that
    \begin{align}\label{ineq:alternativeOUAGforsmallinputs}
		\norm{y(t,x,u)}_Y 
        \leq \tfrac{\eps}{2} + \gamma\!\paren{e^{-k^*+1}s}
        \leq \tfrac{\eps}{2} + \tfrac{\eps}{2} = \eps.
	\end{align}
    The maximum $\overline \tau = \max_{k \in \braces{1, \dots, k^*}}\braces{\tau_k}$ of finitely many elements exists. Hence, from \eqref{ineq:alternativeOUAGforbiginputs} and \eqref{ineq:alternativeOUAGforsmallinputs}, it follows that for all $x \in B_r$, $u \in B_{s,\U}$ and $t \in I$, $t \geq \overline\tau$, it holds that
    \begin{align*}
		\norm{y(t,x,u)}_Y 
        \leq \eps + \gamma\!\paren{e \cdot \norm u_{\U}}\!,
	\end{align*}
    i.e., $\Sigma$ is OUAG.
\end{proof}

Additionally, we generalize the concept of \emph{complete UAG} introduced in \cite{Mir19a} to systems with outputs.
\begin{definition}
    \label{def:OCAG}
    $\Sigma$ has the \emph{output-complete asymptotic gain property (OCAG)} if there exist $\beta \in \KK\LL$, $\gamma \in \KK_\infty$ and $c \ge 0$ such that $\forall x \in X$, $\forall u \in \U$, the following holds:
    \begin{align*} 
        \norm{y(t,x,u)}_Y \leq \beta\!\paren{\norm{x}_X + c,t} + \gamma\!\paren{\norm{u}_\U}\!, \qquad t \in I.
    \end{align*}
\end{definition}

Clearly, OCAG implies OGUAG $\wedge$ BORS. 
We will show in Proposition \ref{prop:OCAGequivalences} that the converse holds as well.

\subsection{Weak attractivity properties}

Weak attractivity for dynamical systems was introduced in \cite{Bha66}. The limit property (LIM) extends it to control systems with full-state output \cite{SoW96} and is essential for ISS superposition theorems. To characterize ISS for infinite-dimensional systems, several variations of the LIM property have been introduced in \cite{MiW18b}. We extend these notions to systems with outputs.
\begin{definition}\label{def:OLIM}
    $\Sigma$ is said to possess the \emph{output-limit property (OLIM)} if there exists 
	$\gamma \in \KK_\infty$ such that for all $\eps > 0$, all $x \in X$ and all $u \in \U$, there is $t = t(\eps,x,u) \in I$ such that
	\begin{align*} 
		\norm{y(t,x,u)}_Y \leq \eps + \gamma\!\paren{\norm{u}_{\U}}\!.
	\end{align*}
\end{definition}
In other words, system $\Sigma$ is OLIM, if for any input $u$ and any initial state, its output approaches the ball of radius $\gamma\!\paren{\norm{u}_{\U}}$
arbitrarily close.


As shown in \cite[Ex. 1]{MiW18b} for the special case of ISS, OLIM and OL are in general not sufficient to imply IOS for infinite-dimensional systems. 
Therefore, we introduce the following new notions, which are stronger as compared to OLIM.
\begin{definition}\label{def:OGULIM}
    We say $\Sigma$ possesses the \emph{output-global uniform limit property (OGULIM)} if there exists
	$\gamma \in \KK_\infty$ such that for all $\eps, r > 0$, there exists $\tau = \tau(\eps,r) \in I$ such that for all $x \in B_r$ and all $u \in \U$, there exists $t \in I$, $t \leq \tau$ such that
	\begin{align*} 
		\norm{y(t,x,u)}_Y \leq \eps + \gamma\!\paren{\norm{u}_{\U}}\!.
	\end{align*}
\end{definition}

\begin{definition}\label{def:OULIM}
    We say $\Sigma$ possesses the \emph{output-uniform limit property (OULIM)} if there exists
	$\gamma \in \KK_\infty$ such that for all $\eps, r, s > 0$, there exists $\tau = \tau(\eps,r,s) \in I$ such that for all $x \in B_r$ and all $u \in B_{s,\U}$, there exists $t \in I$, $t \leq \tau$ such that
	\begin{align*} 
		\norm{y(t,x,u)}_Y \leq \eps + \gamma\!\paren{\norm{u}_{\U}}\!.
	\end{align*}
\end{definition}

In the case of OLIM, the approaching speed towards the ball of radius $\gamma\!\paren{\norm{u}_{\U}}$ depends on the input and the initial state. For OULIM, this speed only depends on the norm of the input and the initial state. And in the case of OGULIM, the speed of approach is also uniform in the input and does only depend on the norm of the initial state.


In the following, we provide an equivalent characterization of OULIM and show that OULIM and OGULIM are equivalent if $h$ is bounded on bounded sets.
    


    
    
    

\begin{lemma}\label{lem:OULIMtoOGULIM}
    Let $\Sigma = (I, X, \U, \phi, Y,h)$ be a forward complete control system with outputs, then, $\Sigma$ is OULIM if and only if there exists
	$\gamma \in \KK_\infty$ such that for all $\eps, r, s > 0$, there exists $\tau = \tau(\eps,r,s) \in I$ such that for all $x \in B_r$ and all $u \in B_{s,\U}$, there exists $t \in I$, $t \leq \tau$ such that
	\begin{align}\label{ineq:alternativeOULIM}
		\norm{y(t,x,u)}_Y \leq \eps + \gamma\!\paren{s}\!.
	\end{align}
    
    Furthermore, if $h$ is bounded on bounded sets, then OULIM$\iff$OGULIM.
\end{lemma}



\begin{proof}
    The equivalence of the two characterizations of OULIM is completely analogous to the one of Lemma \ref{lem:alternativeOUAG}, except that "for all $t \in I$, $t \geq \tau$" must be exchanged by "there exists $t \in I$, $t \leq \tau$".
    
    OGULIM$\implies$OULIM: follows directly from the definition of these concepts. 
    
    OULIM$\implies$OGULIM: 
    Let $\Sigma$ be OULIM with $\gamma, \tau$ as in Definition \ref{def:OULIM} and $h$ be bounded on bounded sets with parameters $\sigma_1, \gamma_1$ and $c > 0$ as in Definition \ref{def:hBoundedOnBoundedSets}. Fix $\eps,r > 0$, take any $x \in B_r$, any $u \in \U$, and let $R \coloneq \gamma^{-1}(\sigma_1(r) + c)$. 
    
    If $\norm{u}_\U \geq R$, then 
    \begin{align*}
        \norm{y(0,x,u)}_Y &= \norm{h(x,u)}_Y
        \leq \sigma_1\paren{\norm x_X} + \gamma_1\!\paren{\norm u_{\U}} + c \\
        &\leq \sigma_1\paren{r} + c + \gamma_1\!\paren{\norm u_{\U}}
        \leq (\gamma + \gamma_1)\!\paren{\norm u_{\U}}\!.
    \end{align*}
    Conversely, if $\norm{u}_\U \leq R$, then by OULIM there exists $t \in I$, $t \leq \tau\!\paren{\eps , r, R}$ such that
    \begin{align*}
        \norm{y(t,x,u)}_Y
        \leq \eps + \gamma\!\paren{\norm{u}_\U}\!.
    \end{align*}
    Then, $\widetilde \tau\!\paren{\eps , r} = \max\{ \tau\!\paren{\eps , r, R}, 0 \} = \tau\!\paren{\eps , r, R}$ is an upper time bound for the OGULIM behavior that does not depend on $\norm{u}_\U$ such that
    \begin{align*}
        \norm{y(t,x,u)}_Y
        \leq \eps + \widetilde \gamma\!\paren{\norm{u}_\U}
    \end{align*}
    for $\widetilde \gamma \coloneq \gamma + \gamma_1 \in \KK_\infty$.
    Hence, $\Sigma$ is OGULIM.
\end{proof}


\begin{remark}
    Lemma \ref{lem:OULIMtoOGULIM} is even new for systems with full-state output, though property \eqref{ineq:alternativeOULIM} already appeared earlier in \cite[eq. (2.84)]{Mir23}.
    \hspace*{\fill}~\QED
\end{remark}




\begin{table}
    \centering
    \caption{List of system properties and abbreviations}
    \begin{tabular}{lll}
        Abbr. & Property & Def. \\\hline
        BORS & bounded output reachability sets  & \ref{def:BORS} \\
        FC & forward completeness & \ref{def:forwardComplete} \\
        IOpS & input-to-output practical stability &\ref{def:IOpS} \\
        IOS & input-to-output stability  & \ref{def:IOS} \\
        IOSS & input/output-to-state stability & \ref{def:IOSS} \\
        ISS & input-to-state stability  & \ref{def:ISS} \\
        local OL & local output Lagrange stability  & \ref{def:OL} \\
        OAG & output-asymptotic gain property  & \ref{def:OGUAG} \\
        OBORS & output-bounded output reachability sets & \ref{def:OBORS} \\
        OCAG & output-complete asymptotic gain property &\ref{def:OCAG} \\
        OCEP & output continuity at the equilibrium point & \ref{def:OCEP}\\
        OGUAG & output-global uniform asymptotic gain property  & \ref{def:OGUAG} \\
        OGULIM & output-global uniform limit property & \ref{def:OGULIM} \\
        OL & output Lagrange stability  & \ref{def:OL} \\
        OLIM & output-limit property & \ref{def:OLIM} \\
        OOUGB & output-to-output-uniform global boundedness & \ref{def:OOUGB} \\
        OOULIM & output-to-output uniform limit property & \ref{def:OOULIM} \\
        OUAG & output-uniform asymptotic gain property  & \ref{def:OGUAG} \\
        OUGB & output-uniform global boundedness  & \ref{def:OULS} \\
        OUGS & output-uniform global stability  & \ref{def:OULS} \\
        OULIM & output-uniform limit property & \ref{def:OULIM} \\
        OULS & output-uniform local stability  & \ref{def:OULS}
    \end{tabular}
    \label{tab:Abbreviations}
\end{table}

\section{Superposition theorems}\label{sec:superpositionThms}

\input{images/implicationDiagram}

The main result of this paper is summarized in Figure~\ref{fig:mainResult}. First, we establish several equivalent characterizations of IOS in Theorem \ref{thm:IOSequivalences}. The course of action is depicted in Figure~\ref{fig:implicationDiagramIOS}. In Proposition \ref{prop:OCAGequivalences}, we give a superposition theorem for OCAG in terms of OUAG  and OGUAG, respectively. Then, we will show equivalences for IOS $\land$ OL in Proposition~\ref{prop:OLIOSequivalences}. By Example \ref{ex:IOSnottoOL}, it becomes clear that the notions of IOS and OL are independent of each other. Furthermore, from Lemma~\ref{lem:IOStoBORSandOCEPandOGUAGandOGULIM}, it follows that IOS implies OULIM~$\land$~OUGS, but the converse implication does not hold true in general as explained in Example \ref{ex:OGULIMandOUGSnottoIOS}.

\subsection{IOS superposition theorem}

We start by stating the following characterization of IOS.
\begin{theorem}[IOS superposition theorem]
\label{thm:IOSequivalences}
    Let $\Sigma = (I, X, \U, \phi, Y,h)$ be a forward complete control system with outputs. Then, the following statements are equivalent:
    \begin{enumerate}
        \item $\Sigma$ is IOS. \label{itm:thmEquivStatement1}
        \item $\Sigma$ is OUAG, OCEP and BORS. \label{itm:thmEquivStatement2}
        \item $\Sigma$ is OUAG, OULS and BORS. \label{itm:thmEquivStatement3}
        \item $\Sigma$ is OUAG and OUGS. \label{itm:thmEquivStatement4}
        \item $\Sigma$ is OCAG and OULS.
        \label{itm:thmEquivStatement5}
        \pb{\item $\Sigma$ is OCAG and OCEP.
        \label{itm:thmEquivStatement6}}
        %
    \end{enumerate}
\end{theorem}

\begin{proof}
    \pb{\input{images/implicationDiagramIOS_new}}
    \pb{We prove the Theorem as depicted in Figure~\ref{fig:implicationDiagramIOS}.}
    The \pb{implications \ref{itm:thmEquivStatement1}$\implies$\ref{itm:thmEquivStatement5} and \ref{itm:thmEquivStatement5}$\implies$\ref{itm:thmEquivStatement2} follow} from Lemma~\ref{lem:IOStoBORSandOCEPandOGUAGandOGULIM}. Lemma \ref{lem:OUAGandOCEPtoOULS} gives the implication \ref{itm:thmEquivStatement2}$\implies$\ref{itm:thmEquivStatement3}. From Proposition \ref{prop:OCAGequivalences}, 
    we have \ref{itm:thmEquivStatement3}$\implies$\ref{itm:thmEquivStatement5}.
    For the implication \ref{itm:thmEquivStatement5}$\implies$\ref{itm:thmEquivStatement1}, we use Proposition \ref{prop:OCAGequivalences} and Lemma \ref{lem:OUGBandOULStoOUGS} to achieve
    \begin{align*}
        \text{OCAG} \land \text{OULS} 
        \implies \text{OUGB} \land \text{OULS}
        \implies \text{OUGS}.
    \end{align*}
    OCAG $\land$ OUGS$\implies$IOS follows from Lemma \ref{lem:OCAGandOUGStoIOS} and closes the circle of equivalences.
    Finally,  \pb{\ref{itm:thmEquivStatement2}$\iff$\ref{itm:thmEquivStatement6} holds true by Prop \ref{prop:OCAGequivalences}.}
\end{proof}

Next, we present the technical lemmas, which we use in the proof of Theorem~\ref{thm:IOSequivalences}.
\begin{lemma}\label{lem:IOStoBORSandOCEPandOGUAGandOGULIM}
    Let $\Sigma = (I, X, \U, \phi, Y,h)$ be a forward complete control system with outputs. Then, the implications depicted in Figure \ref{fig:IOSimplications} hold true.
    \input{images/IOSimplications}

\end{lemma}



\begin{proof}
    IOS$\implies$OUGS:\label{lem:IOStoOUGS} Let $\Sigma$ be IOS. Then, there exist $\beta \in \KK\LL$ and $\gamma \in \KK_\infty$ such that $\forall x \in X$, $\forall u \in \U$ and $\forall t \in I$ the following holds:
    \begin{align*}
        \norm{y(t,x,u)}_Y \leq \beta\!\paren{\norm{x}_X,t} + \gamma\!\paren{\norm{u}_\U}\!.
    \end{align*}
    We can estimate $\beta(\norm{x}_X,t) \leq \beta\!\paren{\norm{x}_X,0} \eqcolon \sigma(\norm{x}_X)$ for all $x \in X$, $t \in I$ and $\sigma \in \KK_\infty$. Then, $\Sigma$ is OUGS.
    
    %
    
    OULS$\implies$OCEP:\label{lem:OULStoOCEP}
    This is a direct consequence of Lemma~\ref{lem:redefineOLOULS}.
    

    IOS$\implies$OCAG:\label{lem:IOStoOCAG} The implication follows immediately from setting $c = 0$ in the definition of OCAG.
    
    OCAG$\implies$OGUAG:\label{lem:OCAGtoOGUAG} Let $\Sigma$ be OCAG. We show OGUAG: For all $\eps,r > 0$, the map $\tau$ can be chosen by
    \begin{align*}
        \tau(\eps,r) = \min\!\braces{t \in I\,\middle|\, \beta(r + c,t) \leq \eps}\!.
    \end{align*}
    As $\beta$ is strictly decreasing to zero in $t$, this minimum exists.
    
    
    Remaining implications: Clear.
\end{proof}



\begin{lemma}
\label{lem:OUAGandOCEPtoOULS}
    Let $\Sigma = (I, X, \U, \phi, Y,h)$ be forward complete. If $\Sigma$ is OUAG and OCEP, then it is OULS.
\end{lemma}

\begin{proof}
    Let $\eps > 0$. We choose $r,s = 1$ and define $T \coloneq \tau\!\paren{\frac{\eps}{2},r,s}$. By OUAG, for every $x \in B_r$, $u\in B_{s,\U}\colon \norm{u}_\U \leq \gamma^{-1}\!\paren{\frac{\eps}{2}}$, and $t \in I\colon\ t \geq T$, it holds that
    \begin{align*} 
    	\norm{y(t,x,u)}_Y \leq \tfrac{\eps}{2} + \gamma\!\paren{\norm{u}_{\U}} \leq \eps.
	\end{align*}
    By OCEP, there exists $\delta = \delta(\eps, T)$ such that the implication
    \begin{align*} 
		t \in I\colon t \leq T, \, \norm{x}_X \leq \delta, \, \norm{u}_{\U} \leq \delta \implies \norm{y(t,x,u)}_Y \leq \eps
	\end{align*}
    holds. By choosing $\widetilde \delta = \min\!\braces{\delta, 1,\gamma^{-1}\!\paren{\frac{\eps}{2}}}$, it follows that
    \begin{align*}
        \norm{x}_X \leq \widetilde\delta, \ \norm{u}_{\U} \leq \widetilde\delta,\  t \in I \ \implies \ \norm{y(t,x,u)}_Y \leq \eps,
    \end{align*}
    i.e., $\Sigma$ is OULS.
\end{proof}

Next, we show the following technical result:
\begin{lemma}\label{lem:OUAGandBORStoOUGB}
    Let $\Sigma = (I, X, \U, \phi, Y,h)$ be a forward complete control system with outputs. Let $\Sigma$ be OUAG and BORS. Then, $\Sigma$ is OUGB.
\end{lemma}
\begin{proof}
    The proof is adapted from \cite[Prop. 10]{MiW18b}. 
    First, we define a parameter $\tau$ such that for sufficiently small $x, u$ the output remains bounded for $t \leq \tau$. Let $\widetilde\gamma$ be the OUAG gain.
    Let $r = s > 0$ and $\eps = 1$. By OUAG, there exists \amc{$\tau(r) = \tau(\eps,r,s)$} such that for all $x \in B_r$ and $u \in B_{r,\U}$
    \begin{align}\label{ineq:boundednessLimit}
		\norm{y(t,x,u)}_Y \leq 1 + \widetilde \gamma\!\paren{\norm{u}_{\U}}\!, \qquad t \geq \tau(r)
	\end{align}
    holds. We can choose $\tau$ to be increasing and continuous. If $\tau$ is increasing but not continuous, it is locally Riemann-integrable so we can replace it by the continuous \amc{and still increasing \cite[Prop. 2.54]{Mir23}} function $\overline \tau = \overline \tau(r) \coloneq \frac{1}{r} \int_r^{2r}\tau(s) \diff s \geq \tau(r)$, $r > 0$. 

    By BORS, there exists a continuous and component-wise increasing function $\mu\colon (\R_+)^3 \to \R_+$ that provides the bound
    \begin{align}\label{ineq:BORSfunction}
        \norm{y(t,x,u)}_Y \leq\mu\!\paren{\norm{x}_X\!, \norm{u}_\U\!,t}\!.
    \end{align}
    Existence of such $\mu$ can be proven analogously to \cite[Lem.  2.12]{Mironchenko2018} and is omitted here.
    Then, from \eqref{ineq:BORSfunction}, we have
    \begin{align}\label{impl:localOutputBoundedness}
        x \in B_r,\, u \in B_{r, \U},\, t \leq \tau(r) \!\implies \!\norm{y(t,x,u)}_Y \leq \widetilde \sigma(r),
    \end{align}
    where we define the continuous and increasing function $\widetilde \sigma\colon \R_+ \to \R_+$, $r \mapsto \mu\!\paren{r, r, \tau(r)}$.


    Define $\sigma(s) \coloneq \widetilde \sigma(s) - \widetilde \sigma(0)$, $s\ge 0$. Clearly, $\sigma \in \KK$. Applying \eqref{impl:localOutputBoundedness} with $r \coloneq \max\!\braces{\norm{x}_X\!, \norm{u}_{\U}}$ for $(x,u) \in X \times \U$, we obtain
    \begin{align}
        \norm{y(t,x,u)}_Y 
        &\leq \sigma\!\paren{\max\!\braces{\norm{x}_X, \norm{u}_{\U}}} + \widetilde \sigma(0) \nonumber\\
        &\leq \sigma\!\paren{\norm{x}_X} + \sigma\!\paren{\norm{u}_{\U}} + \widetilde \sigma(0) \label{ineq:outputBoundednessSmallTimes}
    \end{align}
    for all $x \in X$, $u \in \U$, $t \in I \cap [0,\tau(r)]$. Then, we can define $c  \coloneq \max\{\widetilde \sigma(0),1\} > 0$, $\gamma(r) = \max\!\braces{ \widetilde \gamma(r), \sigma(r)}$ and obtain from \eqref{ineq:boundednessLimit} and \eqref{ineq:outputBoundednessSmallTimes} 
    \begin{align*}
        \norm{y(t,x,u)}_Y \leq  \sigma\!\paren{\norm{x}_X} + \gamma\!\paren{\norm{u}_{\U}} + c.
    \end{align*}
    This proves OUGB of $\Sigma$.
\end{proof}



The next result characterizes OCAG and thereby generalizes \cite[Prop. III.4]{Mir19a} to systems with outputs. Even more, it shows the equivalence of OUAG and OGUAG given that $\Sigma$ is BORS and provides sufficient conditions for IOpS.

\amc{For the case of full-state output, in \cite{Mir19a} the ISpS has been characterized by UAG with respect to certain bounded sets. These characterizations are outside of the scope of this paper, and are a nice direction for future research.}

\begin{proposition}
\label{prop:OCAGequivalences}
    Let $\Sigma = (I, X, \U, \phi, Y,h)$ be a forward complete control system with outputs. 
    Then, the following are equivalent:
    \begin{enumerate}
        \item \label{cond:OCAG_1} $\Sigma$ is OUAG and BORS.
        \item \label{cond:OCAG_2} $\Sigma$ is OGUAG and OUGB.
        \item \label{cond:OCAG_3} $\Sigma$ is  OCAG.
    \end{enumerate}
    Any of these properties implies that $\Sigma$ is IOpS.
\end{proposition}




\begin{proof}
    \pb{\input{images/implicationDiagramOCAG}}
    \pb{
    We follow the approach depicted in Fig. \ref{fig:implicationDiagramOCAG}.
    }
   \ref{cond:OCAG_1}$\implies$\ref{cond:OCAG_2}: Let $\Sigma$ be OUAG with functions $\gamma, \tau$ as given in Definition \ref{def:OGUAG}.
    Let $\eps,r > 0$.     
    We define $\tau_1 = \tau_1(\eps,r) \coloneq \tau(\eps,r,\max\{r, 1\})$. By OUAG, it holds that
    \begin{align*} 
        \norm{y(t,x,u)}_Y &\leq \eps + \gamma\!\paren{\norm{u}_{\U}}, \\
        &\qquad x \in B_r, \, \|u\|_{\U} \leq \max\{r, 1\}, \, t \in I\colon t \geq \tau_1.
    \end{align*}

    On the other hand, by Lemma \ref{lem:OUAGandBORStoOUGB}, $\Sigma$ is OUGB with parameters $\sigma$,  $\gamma$ and $c$. W.l.o.g., $\gamma$ is the same as the one from OUAG. Otherwise, we choose the maximum of the two. 

    Now, for all $x \in B_r, \, \|u\|_{\U} \geq \max\{r, 1\}$, we have that $\|u\|_{\U} \geq \max\{\|x\|_X, 1\}$, and
    \begin{align*}
        \norm{y(t,x,u)}_Y &\leq \sigma\!\paren{\norm{x}_X} + \gamma\!\paren{\norm{u}_{\U}} + c \\
        &\leq \eps +(\sigma + \gamma)\!\paren{\norm{u}_{\U}} + c \norm{u}_{\U}, \qquad  t \in I.
    \end{align*}
    
    As a consequence, for all $\eps,r > 0$ and $\tau_1 = \tau_1(\eps, r) \in I$ it holds for all $x \in B_r$, $u\in \U$, and $t \in I\colon t \geq \tau_1$ that
    \begin{align*} 
        \norm{y(t,x,u)}_Y &\leq \eps + (\sigma + \gamma)\!\paren{\norm{u}_{\U}} + c \norm{u}_{\U}
    \end{align*}
    i.e., $\Sigma$ is OGUAG. 

    \ref{cond:OCAG_2}$\implies$\ref{cond:OCAG_3}: Let $\sigma, \gamma \in \KK_\infty$, $c > 0$ be the parameters in the definition of OUGB. W.l.o.g., let $\gamma$ also be the \amc{OGUAG gain} (otherwise define $\gamma$ as the maximum of the OUGB and the OUAG gain). For all $r \geq 0$, we define the sequence $(\eps_n(r))_{n \in \N_0}$ where $\eps_0(r) \coloneq \sigma(r) + r$ and $\eps_n(r) \coloneq e^{-n}\eps_0(r)$ for $n \in \N$. By OGUAG, for every $r>0$ and any $n\in\N_0$ there is $\tau_n \coloneq \tau(\eps_n(r),r) \in I$ such that
    \begin{align}\label{ineq:boundOUAGtoOGUAG}
        \norm{y(t,x,u)}_Y\!
        &\leq \eps_n(r) + \gamma\!\paren{\norm{u}_{\U}} 
        \leq \eps_n(r + c) + \gamma\!\paren{\norm{u}_{\U}}
    \end{align}
    holds true for all $x \in B_r$, $u \in \U$ and $t \in I\colon t \geq \tau_n$. We can set $\tau_0 = 0$, as by OUGB it holds that
    \begin{align*}
        \norm{y(t,x,u)}_Y
        &\leq \sigma\!\paren{\norm{x}_X} + \gamma\!\paren{\norm{u}_{\U}} + c \\
        &\leq \sigma\!\paren{\norm{x}_X + c} + \gamma\!\paren{\norm{u}_{\U}} + \norm{x}_X + c \\
        &\leq \eps_0(r + c) + \gamma\!\paren{\norm{u}_{\U}}
    \end{align*}
    for all $x \in B_r$, $u \in \U$ and $t \in I\colon t \geq \tau_0$.

    For $r,t \geq 0$, we define the $\KK\LL$-function $\beta$ by
    \begin{align*}
        \beta(r,t) = \exp\!\paren{-(n - 1) - \tfrac{t - \tau_n}{\tau_{n + 1} - \tau_n}}\eps_0(r),
    \end{align*}
    which is piecewise defined for $t \in [\tau_n,\tau_{n + 1})$, $n \in \N_0$.
        
    From this construction, we obtain
    \begin{align*}
        \beta(r,t) \geq e^{-n}\eps_0(r) = \eps_n(r), \qquad t \in [\tau_n,\tau_{n + 1}),\, n \in \N_0.
    \end{align*}
    This especially holds for $r = \norm{x}_X + c$ and by \eqref{ineq:boundOUAGtoOGUAG}, we obtain
    \begin{align*}
        \norm{y(t,x,u)}_Y 
        &\leq \eps_n(\norm{x}_X + c) + \gamma\!\paren{\norm{u}_\U}\\
        &\leq \beta\!\paren{\norm{x}_X + c,t} + \gamma\!\paren{\norm{u}_\U}\!,
    \end{align*}
    for $t \in [\tau_n,\tau_{n + 1})$, $n \in \N_0$,
    i.e., OCAG of $\Sigma$, as desired.
    
    \ref{cond:OCAG_3}$\implies$\ref{cond:OCAG_1}: 
    The implication follows from Lemma \ref{lem:IOStoBORSandOCEPandOGUAGandOGULIM}.
    
    \ref{cond:OCAG_3}$\implies$IOpS: OCAG implies for all $x \in X$, $u \in \U$ and $t \in I$ that
\begin{align*}
    \norm{y(t,x,u)}_Y 
    &\leq \beta\!\paren{\norm{x}_X + c,t} + \gamma\!\paren{\norm{u}_\U} \\
    &\leq \beta\!\paren{2\norm{x}_X,t} +\beta\!\paren{2c,t} + \gamma\!\paren{\norm{u}_\U} \\
    &\leq \beta\!\paren{2\norm{x}_X,t} + \gamma\!\paren{\norm{u}_\U} + \widetilde c
\end{align*}
 is satisfied for $\widetilde c \coloneq \beta\!\paren{2c,0}$. This concludes the proof.
\end{proof}



The next result is a superposition theorem for OUGS.
\begin{lemma}\label{lem:OUGBandOULStoOUGS}
    Let $\Sigma = (I, X, \U, \phi, Y,h)$ be a forward complete control system with outputs. Then $\Sigma$ is OUGS if and only if $\Sigma$ is OUGB and OULS.
\end{lemma}

\begin{proof}
    By Lemma \ref{lem:IOStoBORSandOCEPandOGUAGandOGULIM}, OUGS implies OUGB and OULS. 
    We show the converse implication by adapting the argument from \cite[Lem. I.2]{SoW96}.
    By OUGB, there exist $\sigma_1, \gamma_1 \in \KK_\infty$, $c > 0$ such that all $x \in X$ and $u \in \U$ satisfy
	\begin{align}\label{ineq:OUGB}
		\norm{y(t,x,u)}_Y \leq \sigma_1\!\paren{\norm{x}_X} + \gamma_1\!\paren{\norm{u}_{\U}} + c, \qquad t \in I.
	\end{align}
   By OULS, there exist $\sigma_2, \gamma_2 \in \KK_\infty$ and $r > 0$ such that all $x \in B_r$ and $u \in B_{r,\U}$ satisfy
	\begin{align}\label{ineq:OL}
		\norm{y(t,x,u)}_Y \leq \sigma_2\!\paren{\norm{x}_X} + \gamma_2\!\paren{\norm{u}_{\U}}\!, \qquad t \in I.
	\end{align}
    We now choose $\sigma, \gamma \in \KK_\infty$ such that
    \begin{align*}
        \sigma(s) &\geq 
        \begin{cases}
            \max\!\braces{\sigma_1(s),\sigma_2(s)}\!, &\text{if } s < r, \\
            \sigma_1(s) + c, &\text{if } s \geq r,
        \end{cases}
        \\
        \gamma(s) &\geq 
        \begin{cases}
            \max\!\braces{\gamma_1(s),\gamma_2(s)}\!, &\text{if } s < r, \\
            \gamma_1(s) + c, &\text{if } s \geq r.
        \end{cases}
    \end{align*}
    Then, we distinguish three cases: 
    
    Case 1: For $\norm{x}_X\!, \norm{u}_{\U} < r$, from \eqref{ineq:OL}, we obtain
    \begin{align*}
        \norm{y(t,x,u)}_Y \leq \sigma\!\paren{\norm{x}_X} + \gamma\!\paren{\norm{u}_{\U}}\!, \qquad t \in I.
    \end{align*}
    
    Case 2: For $\norm{x}_X < r \leq \norm{u}_{\U}$, we make use of \eqref{ineq:OUGB}, i.e.,
    \begin{align*}
        \norm{y(t,x,u)}_Y 
        &\leq \sigma_1\!\paren{\norm{x}_X} + \paren{\gamma_1\!\paren{\norm{u}_{\U}} + c} \\
        &\leq \sigma\!\paren{\norm{x}_X} + \gamma\!\paren{\norm{u}_{\U}}\!, \qquad t \in I.
    \end{align*}
    
    Case 3: For $\norm{x}_X \geq r$, we apply \eqref{ineq:OUGB} to obtain
    \begin{align*}
        \norm{y(t,x,u)}_Y 
        &\leq \paren{\sigma_1\!\paren{\norm{x}_X} + c} + \gamma_1\!\paren{\norm{u}_{\U}} \\
        &\leq \sigma\!\paren{\norm{x}_X} + \gamma\!\paren{\norm{u}_{\U}}\!, \qquad t \in I.
    \end{align*}
    
    Therefore, $\Sigma$ is OUGS.
\end{proof}
\begin{lemma}\label{lem:OCAGandOUGStoIOS}
    Let $\Sigma = (I, X, \U, \phi, Y,h)$ be a forward complete control system with outputs. Let $\Sigma$ be OCAG and OUGS. Then, $\Sigma$ is IOS.
\end{lemma}
\begin{proof}
    %
    Let $\beta$, $\gamma$, $c$ be the parameters from the definition of OCAG and $\sigma, \gamma$ the rates of OUGS. 
    By OUGS and OCAG, respectively, it holds for all $x \in X$ and all $u \in \U$ that
    \begin{align*}
        \norm{y(t,x,u)}_Y &\leq \widetilde \beta\!\paren{\norm{x}_X\!,t} + \gamma\!\paren{\norm{u}_\U}\!,\quad t \in I, 
    \end{align*}
    where $\widetilde \beta$ is a $\KK\LL$-function defined by
    \begin{align*}
        \widetilde \beta(r,t) \coloneq \min\!\braces{\paren{1+e^{-t}}\sigma(r), \beta(r + c,t)}\!, \quad r,t \geq 0.
    \end{align*}
    Hence, $\Sigma$ is IOS.
\end{proof}


Having completed the proof of the IOS superposition theorem, we proceed to the characterization of the IOS $\wedge$ OL property.
\subsection{IOS $\wedge$ OL superposition theorem}
Opposed to pure IOS in Theorem \ref{thm:IOSequivalences}, OL allows a characterization of IOS in terms of the OULIM property. In this context also see Example \ref{ex:OGULIMandOUGSnottoIOS}.


\begin{proposition}[IOS $\wedge$ OL superposition theorem]
\label{prop:OLIOSequivalences}
    Let $\Sigma = (I, X, \U, \phi, Y,h)$ be a forward complete control system with outputs. 
    Then the following statements are equivalent:
    \begin{enumerate}
        \item\label{itm:OLIOSequivalences1} $\Sigma$ is IOS and OL.
        \item\label{itm:OLIOSequivalences2} $\Sigma$ is OUAG, OL, and $h$ is $\KK$-bounded.
        \item\label{itm:OLIOSequivalences3} $\Sigma$ is  OULIM, OL, and $h$ is $\KK$-bounded.
    \end{enumerate}
\end{proposition}

\begin{proof}
    By Lemma \ref{lem:IOStoBORSandOCEPandOGUAGandOGULIM}, we have IOS$\implies$OUAG$\implies$OULIM. Furthermore, by Lemma \ref{lem:IOStoBORSandOCEPandOGUAGandOGULIM}, IOS$\implies$OUGS and for $t = 0$ OUGS implies that $h$ is a $\KK$-bounded operator. Hence, the implications \ref{itm:OLIOSequivalences1}$\implies$\ref{itm:OLIOSequivalences2} and \ref{itm:OLIOSequivalences2}$\implies$\ref{itm:OLIOSequivalences3}  hold true.

    It remains to show that
    \begin{align*}
        \text{OULIM} \ \land \ \text{OL} \ \land\ h \text{ is $\KK$-bounded} \quad \implies \quad \text{IOS}.
    \end{align*}
    We follow the approach in \cite[Thm. 2]{Bachmann2022a}.
    
    Using first OL, and then $\KK$-boundedness of $h$, we obtain for all $t \in I$, all $x\in X$ and all $u\in\Uc$ that
    \begin{align*}
        \norm{y(t,x,u)}_Y 
        &\leq \sigma\!\paren{\norm{y(0,x,u)}_Y} + \gamma\!\paren{\norm{u}_{\U}} \\
        &\leq \sigma\!\paren{\sigma_1(\norm{x}_X) + \gamma_1(\norm{u}_\U)} + \gamma\!\paren{\norm{u}_{\U}} \\
        &\leq (\sigma \circ 2\sigma_1)(\norm{x}_X) + (\sigma \circ 2\gamma_1 + \gamma)\!\paren{\norm{u}_{\U}}\!,
    \end{align*}
    which shows OUGS.
    
    Let $\Sigma$ be OULIM and OL, and without loss of generality we assume that the  corresponding gain $\gamma\in\KK$ is in both cases the same. Take any $r>0$, and  define the sequence $(\eps_n(r))_{n \in \N_0}$, where $\eps_0(r) \coloneq (\sigma \circ 2\sigma_1 + \sigma \circ 2\gamma_1 + \gamma)(r)$ and $\eps_n(r) \coloneq e^{-n}\eps_0(r)$ for $n \in \N$. By OULIM, for every $r>0$ and any $n\in\N_0$, there is $\tau_n \coloneq \tau(\eps_n(r),r,r) \in I$ such that for all $x \in X$ and all $u \in \U$ for which $r = \max\{\norm{x}_X\!, \norm{u}_\U\}$ and some $t^* \in I$, $t^* \leq \tau_n$,
    \begin{align}
        \norm{y(t^*,x,u)}_Y 
        &\leq \eps_n(r) + \gamma\!\paren{\norm{u}_{\U}} \nonumber\\
        &\leq \eps_n(\norm{x}_X) + \eps_n(\norm{u}_\U) + \gamma\!\paren{\norm{u}_{\U}} \nonumber\\
        &\leq \eps_n(\norm{x}_X) + (\eps_0 + \gamma)\!\paren{\norm{u}_{\U}} \label{ineq:OULIMConvergingBound}
    \end{align}
    holds true. By OUGS, we can take $\tau_0 = \tau(\eps_0(r),r,r) = 0$. We can assume that the sequence $(\tau_n)_{n \in \N_0}$ is strictly increasing and $\lim_{n \to \infty} \tau_n = \infty$. 

    By the cocycle property, for all $t,s \geq 0$, $x \in X$ and $u \in \U$ the identity
    \begin{align}
        y(t + s,x,u)
        &= h(\phi(t + s,x,u),u(t\pb{+s})) \nonumber\\
        &= h\!\paren{\phi\!\paren{s,\phi(t,x,u),u(t + \ph)},u(t + s)} \nonumber\\
        &= y\!\paren{s,\phi(t,x,u),u(t + \ph)} \label{eq:cocyclePropertyOutput}
    \end{align}
    holds. Then, we can estimate the output for all $t \geq t^*$. In the following calculation, we consecutively use \eqref{eq:cocyclePropertyOutput}, OL, the axiom of shift invariance, and \eqref{ineq:OULIMConvergingBound} to obtain for each $x \in B_r$, $u \in B_{r,\U}$ that there exists $t^* \leq \tau_n$ such that for all $t \geq \tau_n$, it holds that
    \begin{align}
        &\norm{y(t,x,u)}_Y
        = \norm{y\!\paren{t- t^*,\phi(t^*,x,u),u(t^* + \ph)}}_Y \nonumber\\
        &\leq \sigma\!\paren{\norm{y(t^*,x,u)}_Y} + \gamma(\|u(t^* + \ph)\|_\U) \nonumber\\
        &\leq \sigma\!\paren{\norm{y(t^*,x,u)}_Y} + \gamma(\|u\|_\U) \nonumber\\
        &\leq \sigma\!\paren{\eps_n(\norm{x}_X) + (\gamma + \eps_0)\!\paren{\norm{u}_\U}} + \gamma\!\paren{\norm{u}_\U} \nonumber \\
        &\leq \sigma(2\eps_n(\norm{x}_X)) + \sigma\!\paren{2 (\gamma + \eps_0)\!\paren{\norm{u}_\U}} + \gamma\!\paren{\norm{u}_\U} \nonumber \\
        &\leq \widetilde \sigma(\eps_n(\norm{x}_X)) + \widetilde \gamma\!\paren{\norm{u}_\U}\!,\label{ineq:estimateIOSfromOGUAG}
    \end{align}
    where we define $\widetilde \sigma = \sigma(2 \ph)$ and $\widetilde \gamma = \sigma \circ (2 (\gamma + \eps_0)) + \gamma$.
    For $r,t \geq 0$, we define the $\KK\LL$-function $\beta$ by
    \begin{align*}
        \beta(r,t) = \widetilde \sigma\!\paren{\exp\!\paren{-(n - 1) - \tfrac{t - \tau_n}{\tau_{n + 1} - \tau_n}}\eps_0(r)}\!,
    \end{align*}
    which is piecewise defined for $t \in [\tau_n,\tau_{n + 1})$, $n \in \N_0$.
        

    Now, with analogous arguments as in the proof of Proposition \ref{prop:OCAGequivalences}, for all $x \in X$ and all $u \in \U$, it follows that
    \begin{align*}
        \norm{y(t,x,u)}_Y 
        &\leq \beta\!\paren{\norm{x}_X\!,t} + \widetilde \gamma\!\paren{\norm{u}_\U}\!, \quad  t \in I,
    \end{align*}
    as desired.
\end{proof}

\subsection{Sufficient condition for OL}

As OL plays an important role in Proposition \ref{prop:OLIOSequivalences}, we derive sufficient conditions for the OL property.
To this aim, we introduce a modified version of OULIM and OGULIM. The difference between the newly defined OOULIM as compared to OULIM and OGULIM lies in the choice of the \pb{uniformity} with respect to the initial condition. For OOULIM, the initial condition $x$ is chosen such that the output $y(0,x,u)$ is in a bounded ball whereas for OULIM and OGULIM the initial state $x$ itself is bounded. Similarly, we modify BORS.
\begin{definition}\label{def:OOULIM}
    We say $\Sigma$ possesses the \emph{output-to-output uniform limit property (OOULIM)} if there exists
	$\gamma \in \KK_\infty$ such that for all $\eps, r, s > 0$ there exists $\tau = \tau(\eps,r) \in I$ such that for all $x \in X$ and all $u \in B_{s,\U}$ such that $y(0,x,u) \in B_{r,Y}$, there exists $t \in I$, $t \leq \tau$ satisfying
	\begin{align*} 
		\norm{y(t,x,u)}_Y \leq \eps + \gamma\!\paren{\norm{u}_{\U}}.
	\end{align*}
\end{definition}

Note that opposed to OULIM and OGULIM, there is no local and global version of OOULIM with respect to the input as can be seen by the following lemma.
\begin{lemma}\label{lem:OOULIMglobal}
    Let $\Sigma = (I, X, \U, \phi, Y,h)$ be a forward complete control system with outputs. Let $\Sigma$ be OOULIM with parameters $\eps,r,s>0$  and $\tau = \tau(\eps,r,s)$ as in Definition \ref{def:OOULIM}. Then, $\tau$ can be chosen uniformly for all $s$, i.e., there exists $\gamma \in \KK_\infty$ such that for all $\eps, r > 0$, there exists $\tau = \tau(\eps,r)$ such that for all $x \in X$ and all $u \in \U$ such that $y(0,x,u) \in B_{r,Y}$, there exists $t \in I$, $t \leq \tau$ satisfying
	\begin{align*} 
		\norm{y(t,x,u)}_Y \leq \eps + \gamma\!\paren{\norm{u}_{\U}}\!.
	\end{align*}
\end{lemma}
\begin{proof}
    Let $\Sigma$ be OOULIM with $\gamma, \tau$ as in Definition \ref{def:OOULIM} and fix $\eps,r > 0$. Let $R = R(r) \coloneq \gamma^{-1}(\max\{r - \eps,0\})$. Then, for $x \in X$, $u \in \U$, such that $\norm{u}_\U \geq R$, $y(0,x,u) \in B_{r,Y}$ and $t = 0$, it holds that
    \begin{align*}
        \norm{y(0,x,u)}_Y &\leq \eps + r - \eps 
        \leq \eps + \gamma\!\paren{R}
        \leq \eps + \gamma\!\paren{\norm{u}_\U}\!.
    \end{align*}
    Next, we give an estimate for $u \in B_{R,\U}$: By OOULIM, there exists $\tau = \tau\!\paren{\eps , r, R}$ such that for all $x \in X$, $u \in B_{R,\U}$ such that $y(0,x,u) \in B_{r,Y}$, there exists $t \in I$, $t \leq \tau$ such that
    \begin{align*}
        \norm{y(t,x,u)}_Y
        \leq \eps + \gamma\!\paren{\norm{u}_\U}\!.
    \end{align*}
    Then, $\widetilde \tau\!\paren{\eps , r} = \max\{ \tau\!\paren{\eps , r, R}, 0 \} = \tau\!\paren{\eps , r, R}$ is an upper time bound for the OOULIM behavior that does not depend on $\norm{u}_\U$ as we chose $R$ to be a function of $r$.
\end{proof}

\begin{definition}\label{def:OBORS}
    System $\Sigma$ is said to have \emph{output-bounded output reachability sets (OBORS)} if for all $C > 0$ and $\tau \in I$ it holds that
	\begin{align*} 
		\sup_{x \in X,\ \norm{u}_\U,\, \norm{y(0,x,u)}_Y < C, \ t < \tau}\norm{y(t,x,u)}_Y < \infty.
	\end{align*}
\end{definition}

\begin{definition}\label{def:OOUGB}
	We call $\Sigma$ \emph{output-to-output-uniformly globally bounded (OOUGB)} if there exist $\sigma, \gamma \in \KK_\infty$, $c > 0$ such that all $x \in X$ and $u \in \U$ satisfy
	\begin{align*} 
		\norm{y(t,x,u)}_Y \leq \sigma\!\paren{\norm{y(0,x,u)}_Y} + \gamma\!\paren{\norm{u}_{\U}} + c, \quad t \in I.
	\end{align*}
\end{definition}

We are now ready to provide a sufficient condition for OL. The notion of OOULIM is crucial in the following and cannot be exchanged even by OGULIM as shown in Example \ref{ex:OGULIMandLocalOLandOBORSnottoGlobalOL}.
\begin{lemma}\label{lem:OOULIMandLocalOLandOBORStoGlobalOL}
    Let $\Sigma = (I, X, \U, \phi, Y,h)$ be a forward complete control system with outputs. Let $\Sigma$ be OOULIM, locally OL and OBORS. Then, $\Sigma$ is  OL.
\end{lemma}

\begin{proof}
    We follow the approach in \cite[Thm. 5]{MiW18b}.

    Step 1: OOULIM $\land$ OBORS$\implies$OOUGB.
    The proof is adapted from \cite[Prop. 10]{MiW18b}. 
    
    By OBORS, there exists a continuous and 
    with respect to all variables \pb{increasing} function $\mu\colon (\R_+)^3 \to \R_+$ that provides the bound
    \begin{align*}
        \norm{y(t,x,u)}_Y \leq\mu\paren{\norm{y(0,x,u)}_Y\!, \norm{u}_\U\!,t\vphantom{\big(}}\!.
    \end{align*}
    For $r > 0$,  we define $R(r) \coloneq \mu(r,r,1)$.
    
    Next, we define a parameter $\tau$ such that for sufficiently small $x, u$ the output remains bounded for some $t \leq \tau$.
    Let $r = s > 0$ and $\eps = \tfrac{r}{2}$. By OOULIM there exists $\tau = \tau(r) \coloneq \tau\Big(1,\max\!\braces{R(r), \gamma^{-1}\!\paren{\tfrac{r}{2}}}\!\Big)$ such that for all $x \in X$ and $u \in B_{r,\U}$, such that $y(0,x,u) \in B_{R(r),Y}$ and $\norm{u}_\U < \gamma^{-1}\!\paren{\tfrac{r}{2}}$, there exists $t \in I$, $t \leq \tau$ such that
    \begin{align} \label{ineq:boundednessLimitOL}
		\norm{y(t,x,u)}_Y \leq \tfrac{r}{2} + \gamma\!\paren{\norm{u}_{\U}} \leq r
	\end{align}
    holds true. We can assume that $\tau$ is increasing and continuous. If $\tau$ is not continuous, we can replace it by the continuous \amc{and increasing} function $\overline \tau = \overline \tau(r) \coloneq \frac{1}{r} \int_r^{2r}\tau(s) \diff s \geq \tau(r)$.

    With the definition $\widetilde \gamma(r) \coloneq \max\{r, 2\gamma(r)\}$, we have that
    \begin{align}
        &x \in X, u \in \overline{B_{\widetilde \gamma^{-1}(r), \U}}: y(0,x,u) \in \overline{B_{R(r), Y}}, t \leq \tau(r) \nonumber\\ 
        &\implies \norm{y(t,x,u)}_Y \leq \widetilde \sigma(r), \label{impl:outputBoundednessOL}
    \end{align}
    where we define the continuous and increasing function $\widetilde \sigma\colon \R_+ \to \R_+$, $r \mapsto \mu\!\paren{R(r), \widetilde \gamma^{-1}(r), \tau(r)}$. W.l.o.g., it holds that $\widetilde \sigma(r) > R(r)$ for all $r \geq 0$.

    We now use a bootstrapping argument to show that $\norm{y(t,x,u)}_Y \leq \widetilde \sigma(r)$ holds for all $ t \in I$. We assume the converse: Let there exist $x \in X$, $u \in B_{\widetilde \gamma^{-1}(r), \U}$, where $y(0,x,u) \in B_{r, Y}$, $t \in I$ such that $\norm{y(t,x,u)}_Y > \widetilde \sigma(r)$. Furthermore, we define
    \begin{align*}
        t_m \coloneq \sup\!\braces{s \in [0,t]\,\middle|\,\norm{y(s,x,u)}_Y \leq R(r)\vphantom{\big(} } \geq 0.
    \end{align*}
    \pb{Note that $t_m$ is well defined as $\norm{y(0,x,u)}_Y < r \leq R(r)$ by the definition of $R$.}
    By the cocycle property \ref{cond:cocycleProperty} it holds that
    \begin{align*}
        y(t,x,u) = y\!\paren{t - t_m, \phi(t_m,x,u),u(\ph + t_m)\vphantom{\big(}}\!.
    \end{align*}
    
    Case 1: Assume that $t-t_m \leq \tau(r)$ holds. Then, from $\norm{y(t_m,x,u)}_Y \leq R(r)$ and \eqref{impl:outputBoundednessOL}, it follows that $\norm{y(s,x,u)}_Y \leq \widetilde \sigma(s)$ for all $s \in [t_m,t]$.

    Case 2: Assume $t-t_m > \tau(r)$. Then by \eqref{ineq:boundednessLimitOL}, there exists some $\pb{t_m} \leq t^* < \tau(r) \pb{+ t_m}$ such that
    \begin{align*}
        \norm{y(t^*,x,u)}_Y &= \norm{y\!\paren{t^*\pb{- t_m},\phi(t_m,x,u), u(\ph + t_m)\vphantom{\big(}}}_Y 
        \leq r.
    \end{align*}
    But then, for all $s \in \big(t^*,\min\{t^*+1,t\}\big]$, it holds that
    \begin{align*}
        \norm{y(s,x,u)}_Y &= \norm{y\!\paren{s - t^*, \phi(t^*,x,u),u(\ph + t^*)\vphantom{\big(}}} \\
        &\leq R(\norm{y(t^*,x,u)}_Y) 
        \leq R(r),
    \end{align*}
    which contradicts the definition of $t_m$ \pb{as $s > t_m$}.
    Therefore, 
    \begin{align}
        &x \in X,\ u \in \overline{B_{\widetilde \gamma^{-1}(r), \U}}: y(0,x,u) \in \overline{B_{R(r), Y}},\ t \in I \nonumber\\
        &\implies \norm{y(t,x,u)}_Y \leq \widetilde \sigma(r). \label{impl:globalOutputBoundedness}
    \end{align}
    We define $\sigma \in \KK_\infty$, $\sigma(r) \coloneq \widetilde \sigma(r) - \widetilde \sigma(0)$. Let us define $r \coloneq \max\!\braces{\norm{y(0,x,u)}_Y\!, \widetilde\gamma(\norm{u}_{\U})}$ for $(x,u) \in X \times \U$. From \eqref{impl:globalOutputBoundedness}, we obtain
    \begin{align*}
        \norm{y(t,x,u)}_Y 
        &\leq \sigma\!\paren{\max\!\braces{\norm{y(0,x,u)}_Y\!, \widetilde\gamma(\norm{u}_{\U})}\vphantom{\big(}\!} + \widetilde \sigma(0) \\
        &\leq \sigma\!\paren{\norm{y(0,x,u)}_Y} + \sigma\!\paren{\widetilde\gamma(\norm{u}_{\U})} + \widetilde \sigma(0)
    \end{align*}
    for all $x \in X, u \in \U, t \in I$. This proves OOUGB of $\Sigma$.
    
    Step 2: OOUGB $\land$ local OL$\implies$OL. The proof is completely analogous to the proof of Lemma \ref{lem:OUGBandOULStoOUGS}, except of the right-hand side of the estimates $y(0,x,u)$ being used instead of $x$.
\end{proof}

\section{finite-dimensional IOS theory}\label{sec:finiteDimIOSTheory}



The present paper is motivated by the IOS superposition theorem for ODE systems proved in \cite[Thm. 1]{ISW01}. We want to rederive these results from our findings.

We consider a finite-dimensional output system
\begin{align}
    \Sigma_\text{finite}\colon 
    \begin{cases}
        \dot x = f(x,u), &t \in I, \\
        y = h(x),
    \end{cases}
\end{align}
where $I = \R_+$, $x \in X = \R^n$, $u \in \U$ and $\U$ is the space of essentially bounded Lebesgue measurable functions from $I$ to $U = \R^m$, \pb{$Y = \R^p$}, for some natural numbers $m,n,p$. Moreover, $f\colon X \times U \to X$ is Lipschitz continuous in the first variable on bounded subsets, i.e., for all $r > 0$, there exists $L(r) > 0$ such that for $x_1,x_2 \in B_r$, $u \in B_{r,U}$, it holds that
\begin{align*}
    \norm{f(x_1,u) - f(x_2,u)}_X \leq L(r)\norm{x_1 - x_2}_X.
\end{align*}
Let $h\colon X \to Y$ be continuous. 









\subsection{OLIM, OULIM and OGULIM on finite-dimensional spaces}
Next, we prove the equivalence of OGULIM, OULIM and OLIM on finite-dimensional spaces.

For a given $x \in X$, a set $\Phi \subset Y$, and an input $u\in\Uc$, let
\begin{align*}
        \tau^o(x, \Phi,u) \coloneq \inf\!\braces{t \geq 0 \,\middle|\, y(t,x,u) \in \Phi}
    \end{align*}
be the \emph{first crossing time} of the set $\Phi$ by the output trajectory $y(\ph,x,u)$.
 
To further proceed, we adapt \cite[Cor. 4.2]{Angeli2004a} (see also \cite[Thm. 1.43]{Mir23}).
\begin{corollary}\label{cor:finiteConvergenceTime}
    Let $\Sigma_\text{finite}$ be a forward complete system.
    Consider
    \begin{itemize}
        \item a compact set $C$ of the state space $X$,
        \item a bounded open neighborhood $\widetilde C$ of $C$,
        \item an open subset $\Phi \subset Y$,
        \item a compact subset $J \subset \Phi$,
        \item a radius $s > 0$
    \end{itemize}
    such that for any $x \in \widetilde C$ and any $u \in \overline{B_{s,\U}}$, there exists $t \in I$ such that $y(t,x,u) \in J$.
		
    Then $\Phi$ can be reached in a uniform time by all trajectories starting in $C$ and under inputs from $\overline{B_{s,\U}}$, that is:
    \begin{align*}
        \sup\nolimits_{x \in C,\ u \in \overline{B_{s,\U}}} \braces{\tau^o(x, \Phi,u)} < \infty.
    \end{align*}       
\end{corollary}

Now we are ready to state the next proposition:

\begin{proposition}
\label{prop:OLIMforFiniteDimensions}
    For finite-dimensional systems $\Sigma_\text{finite}$, the following are equivalent:
    \begin{enumerate}
        \item\label{cond:OLIM_1} $\Sigma_\text{finite}$ is OGULIM.
        \item\label{cond:OLIM_2} $\Sigma_\text{finite}$ is OULIM.
        \item\label{cond:OLIM_3} $\Sigma_\text{finite}$ is OLIM.
    \end{enumerate}
\end{proposition}

\begin{proof}
    \ref{cond:OLIM_1}$\iff$\ref{cond:OLIM_2} is a consequence of Lemma \ref{lem:OULIMtoOGULIM}.
    
    \ref{cond:OLIM_2}$\implies$\ref{cond:OLIM_3} holds by the definitions of OULIM and OLIM.

    We show \ref{cond:OLIM_3}$\implies$\ref{cond:OLIM_2}: Let $\Sigma_\text{finite}$ be OLIM with $\gamma, \tau$ as in Definition \ref{def:OLIM}.
    We fix $r,s, \eps > 0$ and define $\widetilde C \coloneq B_{2r}$, $C \coloneq \overline{B_r}$,
    \begin{align*}
        \Phi &\coloneq \braces{y \in Y \,\middle|\, \norm{y}_Y < \eps + \gamma\!\paren{s}}, \\
        J &\coloneq \braces{y \in Y \,\middle|\, \norm{y}_Y \leq \tfrac \eps 2 + \gamma\!\paren{s}}\!.
    \end{align*}
    Then, for each $x \in \widetilde C$  and $u \in \overline{B_{s,\U}}$, there exists $t = \tau\!\paren{\frac \eps 2, x,u}$ such that $y(t,x,u)\in J$.
    
    By Corollary \ref{cor:finiteConvergenceTime}, there exists a finite time 
    \begin{align*}
        \overline{\tau}(\eps,r,s) \coloneq \sup\nolimits_{x \in \overline{B_r},\ u \in \overline{B_{s,\U}}} \braces{\tau^o(x, \Phi,u)}
    \end{align*}
    such that for all $x \in \overline{B_r}$, $u \in \overline{B_{s,\U}}$, there exists $t \leq \overline{\tau}(\eps,r,s)$ such that $y(t,x,u) \in \Phi$, i.e., $\norm{y(t,x,u)}_Y \leq \eps + \gamma\!\paren{s}$. 
    
    By Lemma \ref{lem:OULIMtoOGULIM}, it follows that $\Sigma_\text{finite}$ is OULIM.
\end{proof}

\begin{remark}
    For systems with full-state output, the terminology ULIM, bULIM and LIM is used instead of OGULIM, OULIM  and OLIM, respectively \cite[Def. 2.47]{Mir23}.    
    \hspace*{\fill}~\QED
\end{remark}

\subsection{IOS on finite-dimensional spaces}
We want to strengthen Theorem \ref{thm:IOSequivalences} and Proposition~\ref{prop:OLIOSequivalences} for ODE systems. We first give a technical lemma.

\begin{lemma}\label{lem:BORSFiniteDim}
    Let $\Sigma_\text{finite}$ be forward complete. Then, $\Sigma_\text{finite}$ is BORS.
\end{lemma}
\begin{proof}
Since $\Sigma_\text{finite}$ is forward complete, by \cite[Prop. 5.1]{Lin1996}, finite-time reachability sets of $\Sigma_\text{finite}$ are bounded. As $h$ is continuous, $\Sigma_\text{finite}$ is BORS.
\end{proof}

\begin{proposition}\label{prop:OUAGtoGOUAGFiniteDim}
    Let $\Sigma_\text{finite}$ be forward complete. $\Sigma_\text{finite}$ is OUAG if and only if it is OGUAG.
\end{proposition}

\begin{proof}
    The claim follows by Lemma~\ref{lem:BORSFiniteDim} and Proposition \ref{prop:OCAGequivalences}.
\end{proof}
In spite of the equivalence between OUAG and OGUAG for finite-dimensional systems shown in Proposition \ref{prop:OUAGtoGOUAGFiniteDim}, the optimal gain for the OUAG property may not be a gain for the OGUAG property, see \cite[Example 2.46]{Mir23} even for ODE systems with full-state output.

Now we are ready to prove \pb{a result similar to \cite[Thm. 1]{ISW01}:}
\begin{proposition}\label{prop:IOSsuperpositionsFiniteDim}
    Let $\Sigma_\text{finite}$ be a forward complete ODE system. Let $h(0) = 0$. Then, each of the following holds true:
    \begin{enumerate}
        \item\label{itm:IOSsuperpositionsFiniteDim1} $\Sigma_\text{finite}$ is OLIM and OL$\iff${}$\Sigma_\text{finite}$ is OGUAG and OL.
        \item\label{itm:IOSsuperpositionsFiniteDim2} If $\Sigma_\text{finite}$ satisfies $f(0,0)  = 0$, then it is OUAG if and only if it is IOS.
    \end{enumerate}
\end{proposition}




\begin{proof}
    We start with \ref{itm:IOSsuperpositionsFiniteDim1}: As $h$ is $\KK$-bounded by continuity of $h$ and $h(0) = 0$, Proposition \ref{prop:OLIOSequivalences} implies OUAG $\land$ OL$\iff$ OULIM $\land$ OL. The equivalence of OLIM and OULIM follows by Proposition \ref{prop:OLIMforFiniteDimensions}.



    Next, we show \ref{itm:IOSsuperpositionsFiniteDim2}: $\Sigma_\text{finite}$ is BORS by Lemma \ref{lem:BORSFiniteDim} 
    and from $f(0,0)  = 0$ follows OCEP: $\phi(\ph,0,0) \equiv 0$ is a trajectory and by \cite[Thm. 1.40]{Mir23}, the trajectories of $\Sigma_\text{finite}$ are Lipschitz continuous with respect to initial states, in particular, for all $r, \tau > 0$ there exists $C = C(r,\tau) > 0$ such that for all $x \in B_r$, $u \in B_{r,\U}$ and $t \in I\colon t \leq \tau$, it holds that
    $\norm{\phi(t,x,u)}_X \leq C \norm{x}_X$.
    OCEP then follows from continuity of $h$. Hence, \ref{itm:IOSsuperpositionsFiniteDim2} follows from Theorem \ref{thm:IOSequivalences}.
    %
    %
    %
    %
\end{proof}





\begin{remark}
    Compared to \cite[Thm. 1]{ISW01}\pb{, Proposition \ref{prop:IOSsuperpositionsFiniteDim}} holds for a more general class of systems as we do not require $f$ to be locally Lipschitz continuous (in both variables) but only Lipschitz continuous in the first variable on bounded subsets and 
    \pb{OUAG is sufficient to imply IOS as opposed to OGUAG in \cite[Thm. 1]{ISW01}.}



    
    Note that Theorem \ref{thm:IOSequivalences} can be applied to forward complete ODE systems whose $f$ is not necessarily Lipschitz continuous. However, in this case the simplifications which we have in Proposition \ref{prop:IOSsuperpositionsFiniteDim} 
    do not occur in general.
    \hspace*{\fill}~\QED
\end{remark}

\section{Characterizations of ISS}
\label{sec:CharacterizationOfISS}

\subsection{ISS superposition theorem}

As a corollary of Theorem \ref{thm:IOSequivalences}, we obtain the ISS superposition theorem proved in \cite[Thm. 5]{MiW18b}. \pb{By this, we show that Theorem \ref{thm:IOSequivalences} and Proposition \ref{prop:OLIOSequivalences} provide a strict generalization of the results for systems with full-state output. Moreover, the proof demonstrates how several stability notions simplify in this special case.} We refer to \cite{MiW18b} for the definitions of the corresponding notions.
\begin{corollary}[ISS superposition theorem]\label{cor:ISSsuperpositions}
    Consider a system $\Sigma$ with full-state output. 
    Then the following statements are equivalent:
    \begin{enumerate}
        \item\label{itm:corISSsup1} $\Sigma$ is ISS.
        \item\label{itm:corISSsup2} $\Sigma$ is UAG $\land$ CEP $\land$ BRS.
        \item\label{itm:corISSsup3} $\Sigma$ is ULIM $\land$ UGS.
        \item\label{itm:corISSsup4} $\Sigma$ is ULIM $\land$ ULS $\land$ BRS.
    \end{enumerate}
\end{corollary}

\begin{proof}
    Theorem \ref{thm:IOSequivalences} states the equivalence ISS$\iff$UAG $\land$ CEP $\land$ BRS as all of these notions for systems with outputs reduce accordingly. 
    
    Next, OL defines stability on the output-value space which is equivalent to UGS for systems with full-state output. As ISS already implies UGS, Proposition \ref{prop:OLIOSequivalences} \pb{strictly generalizes} the equivalence ISS$\iff$ULIM $\land$ UGS to systems with outputs. 
    
    By Lemma \ref{lem:OOULIMandLocalOLandOBORStoGlobalOL}, we have
    \begin{center}
        OOULIM $\land$ local OL $\land$ OBORS$\implies$OL,
    \end{center}
    which for systems with full-state output reads precisely as ULIM $\land$ ULS $\land$ BRS$\implies$UGS.
    The converse implication UGS$\implies$ULS $\land$ BRS follows from Lemma \ref{lem:IOStoBORSandOCEPandOGUAGandOGULIM}. This shows the equivalence \ref{itm:corISSsup3}$\iff$\ref{itm:corISSsup4}.
\end{proof}







\subsection{ISS as superposition of IOS and IOSS}\label{subsec:ISSbyIOSandIOSS}

Another key ISS-like property for systems with outputs is related to the notions of nonlinear detectability of control systems and to the following question: Given the past input and output signals of a system, is it possible to recover information about the current state of the system? Obtaining an estimate of the state in terms of input and output is highly relevant for controller design \cite{Praly1996}.
\pb{System $\Sigma$} is called zero-detectable \pb{\cite{Son81}}, if there is $\beta \in \KK\LL$ such that for all $x \in X$ satisfying $y(\ph, x, 0) \equiv 0$ it holds that
\begin{align*}
    \norm{\phi(t,x,0)}_X \leq \beta\!\paren{\norm{x}_X, t}\!,\qquad t\geq 0.
\end{align*}
For linear finite-dimensional systems, this property is equivalent to the classical detectability.
\\
For nonlinear systems, it is desirable to have robustness of this property with respect to variations of the input and output. 
Motivated by the ISS property, we introduce the following concept:
\begin{definition}\label{def:IOSS}
    A control system $\Sigma$ is called \emph{input/output-to-state stable (IOSS)}, if there exist $\beta \in \KK\LL$ and $\gamma_1, \gamma_2 \in \KK$ such that for all $x \in X$, all $u \in \U$, and all $t \in I$ we have 
    \begin{align} \label{ineq:IOSS}
        \norm{\phi(t, x, u)}_X
        &\leq \beta\!\paren{\norm{x}_X, t} + \gamma_1\!\paren{\norm{u|_{[0, t]}}_\U} \nonumber\\
        &\qquad+ \gamma_2\!\paren{\sup\nolimits_{s \in [0,t]}\norm{y(s,x,u)}_Y}.
    \end{align}
\end{definition}


A variant of the IOSS property for input/output systems was considered under the name of \emph{detectability} in \cite{Son89d}. A \emph{practical} counterpart of IOSS was introduced in \cite{JTP94} by the name \emph{strong unboundedness observability}. Taking in this notion the offset constant equal to zero, we obtain precisely the IOSS concept as defined in \cite{SoW97}. %
Several fundamental results in the IOSS theory have been established in \cite{KSW01}. 
IOSS extends zero-detectability in the same way as ISS extends 0-UGAS. From \eqref{ineq:IOSS}, it can be seen that IOSS systems have ISS zero dynamics (dynamics of the system obtained by choosing the input $u$ so that the output $y$ is identically 0).
Furthermore, IOSS is closely related to strict dissipativity, existence of turnpikes, and model-predictive control \cite{HoG19,Gru22}.

\pb{
The next result generalizes the equivalence \cite[Prop. 3.1]{JTP94}
\begin{align*}
    \text{ISS} \iff \text{IOS} \land \text{IOSS}
\end{align*}
for ODE systems to abstract control systems with outputs.
}

\begin{proposition}\label{prop:ISSequivIOSandIOSS}
    Let $\Sigma = (I, X, \U, \phi, Y,h)$ be a forward complete control system with outputs. Then the following statements are equivalent:
    \begin{enumerate}
        \item\label{cond:ISS_1} $\Sigma$ is ISS and $h$ is $\KK$-bounded.
        \item\label{cond:ISS_2} $\Sigma$ is IOS and IOSS.
    \end{enumerate}
\end{proposition}

 


\begin{proof}
    We start with the implication \ref{cond:ISS_1}$\implies$IOS: Let $\Sigma$ be ISS where $\beta \in \KK\LL$, $\gamma \in \KK_\infty$ as in \eqref{ineq:ISS} and $h$ be a $\KK$-bounded output map with functions $\sigma_1, \gamma_1$ as defined in $\eqref{ineq:boundedOutputMap}$. Then, it follows for any $x \in X$, $u \in \U$ and $t \in I$ for the output function
    \begin{align*}
        &\norm{y(t,x,u)}_Y 
        = \norm{h(\phi(t,x,u),u)}_Y \\
        &\qquad\leq \sigma_1(\norm{\phi(t,x,u)}_X) + \gamma_1(\norm u_{\U}) \\
        &\qquad\leq \sigma_1\!\paren{\beta\!\paren{\norm{x}_X\!,t} + \gamma(\norm{u}_{\U})} + \gamma_1(\norm u_{\U}) \\
        &\qquad\leq \sigma_1\!\paren{2\beta\!\paren{\norm{x}_X\!,t}} + \sigma_1\!\paren{2\gamma(\norm{u}_{\U})} + \gamma_1(\norm u_{\U}) \\
        &\qquad= \widetilde \beta\!\paren{\norm{x}_X,t} + \widetilde\gamma(\norm{u}_{\U}),
    \end{align*}
    where we used the $\KK$-boundedness of $h$ in the second line and the ISS property of $\phi$ in the third line. Here, $\widetilde \beta \coloneq \sigma_1 \circ (2 \beta) \in \KK\LL$ and $\widetilde \gamma \coloneq \sigma_1 \circ (2 \gamma) + \gamma_1 \in \KK_\infty$, i.e., $\Sigma$ is IOS.

    ISS$\implies$IOSS: We have
    \begin{align*} 
        &\norm{\phi(t, x, u)}_X \leq \beta\!\paren{\norm{x}_X\!, t} + \gamma\!\paren{\norm{u|_{[0, t]}}_\U} \\
        &\leq \beta\!\paren{\norm{x}_X\!, t} + \gamma\!\paren{\norm{u|_{[0, t]}}_\U} + \gamma_2\!\paren{\sup\nolimits_{s \in [0,t]}\!\norm{y(s,x,u)}_Y\!}
    \end{align*}
    for arbitrary $\gamma_2 \in \KK$.
    The first inequality holds due to the ISS property and causality of $\Sigma$, i.e., that for all $x \in X$, $u \in \U$ and $t \in  I$, it holds that $\phi(t, x, u) = \phi(t, x, u|_{[0,t]})$.

    IOS$\implies$$h$ bounded: Let $\Sigma$ be IOS with $\beta,\gamma$ as in \eqref{ineq:ISS}. For $\sigma \coloneq \beta(\ph,0)$, it holds that
    \begin{align*}
        \norm{h(x,u)}_Y =\norm{y(0,x,u)}_Y 
        &\leq \beta(\norm x_X,0) + \gamma(\norm u_{\U}) \\
        &= \sigma(\norm x_X) + \gamma(\norm u_{\U}),
    \end{align*}
    exactly as claimed.
    
    We show \ref{cond:ISS_2}$\implies$ISS: We follow the approach in \cite[Prop. 3.1]{JTP94}. We start by showing that IOSS and IOS imply uniform global stability (OUGS with the output equal to the state). Let $\Sigma$ be IOSS with $\beta, \gamma_1, \gamma_2$ as in \eqref{ineq:IOSS} and IOS with $\beta, \gamma$ as in \eqref{ineq:IOS}. W.l.o.g., we assume that $\beta$ is the same function for IOS and IOSS. Substitution of \eqref{ineq:IOS} into \eqref{ineq:IOSS} and the axiom of restriction invariance results in
    \begin{align}
        \norm{\phi(t,x,u)}_X
        &\leq \beta\!\paren{\norm{x}_X\!, t} + \gamma_1\!\paren{\norm{u|_{[0, t]}}_\U} \nonumber\\ 
        &\quad + \gamma_2\!\paren{\sup\nolimits_{t \in I}\paren{\beta\!\paren{\norm{x}_X,t} +\gamma\!\paren{\norm{u}_\U}}} \nonumber\\
        &\leq \sigma\!\paren{\norm{x}_X} + \hat \gamma\!\paren{\norm{u}_\U}\!, \label{ineq:IOSimpliesUGS}
    \end{align}
    where $\sigma \coloneq \beta(\ph,0) + \gamma_2(2\beta(\ph,0))$ and $\hat \gamma \coloneq \gamma_1 + \gamma_2 \circ (2\gamma)$.
    Next, we employ the cocycle property \ref{cond:cocycleProperty} of $\Sigma$ to achieve a time-invariant version of \eqref{ineq:IOSS}, i.e.,
    \begin{align}
        \norm{\phi(t,x,u)}_X
        &= \norm{\phi(t - t_0,\phi(t_0,x,u),u(t_0 + \ph))}_X \nonumber\\
        &\leq \beta\!\paren{\norm{\phi(t_0,x,u)}_X\!, t-t_0} + \gamma_1\!\paren{\norm{u|_{[t_0, t]}}_\U} \nonumber\\
        &\quad+ \gamma_2\!\paren{\sup\nolimits_{s \in [t_0,t]}\norm{y(s,x,u)}_Y} \label{ineq:timeInvariantIOSS}
    \end{align}
    for all $x \in X, u \in \U$ and $t,t_0 \in I\colon t \geq t_0$.
    With \eqref{ineq:timeInvariantIOSS} at hand, we are now able to bound for $t_0 \coloneq \frac t 2$
%
    \begin{align*}
        \norm{\phi(t,x,u)}_X
        &\leq \beta\!\paren{\norm{\phi\!\paren{\tfrac t 2,x,u}}_X\!, \tfrac t 2} + \gamma_1\!\paren{\norm{u|_{\brackets{\frac t 2, t}}}_\U} \\
        &\quad+ \gamma_2\!\paren{\sup\nolimits_{s \in \brackets{\frac t 2,t}}\norm{y(s,x,u)}_Y} \\
        &\leq \beta\!\paren{\sigma\!\paren{\norm{x}_X} + \hat \gamma\!\paren{\norm{u}_\U}\!, \tfrac t 2} + \gamma_1\!\paren{\norm{u}_\U}\\
        &\quad+ \gamma_2\!\paren{\beta\!\paren{\norm{x}_X,\tfrac t 2} + \gamma\!\paren{\norm{u}_\U}} \\
        &\leq \widetilde \beta\!\paren{\norm{x}_X,t} + \widetilde \gamma\!\paren{\norm{u}_\U}\!,
    \end{align*}
    where we used the axiom of restriction invariance, 
    \eqref{ineq:IOS} and \eqref{ineq:IOSimpliesUGS} in the second step. Here, we defined $\widetilde \beta(s,t) \coloneq \beta(2\sigma(s), \frac t 2) + \gamma_2\circ(2\beta(s,\frac t 2)) \in \KK\LL$ and $\widetilde \gamma \coloneq \beta(2 \hat \gamma,0) + \gamma_1 + \gamma_2\circ(2\gamma) \in \KK_\infty$. Hence, $\Sigma$ is ISS.
\end{proof}




\section{Counterexamples}
\label{sec:counterexamples}

In this section, we provide several counterexamples to show that certain implications do not hold.
\begin{remark}
    \pb{By} \cite[Ex. 1]{MiW18b}, OL $\land$ OLIM$\nimplies$IOS in general even in the case of full-state output.
    \hspace*{\fill}~\QED
\end{remark}


\begin{example}[IOS$\nimplies$OL]\label{ex:IOSnottoOL}
    Let us consider the following system with a scalar state and output
    \begin{align*}
        \Sigma\colon\quad  \dot x = -x, \quad y(t,x_0) = \sin(\phi(t,x_0)).
    \end{align*}
    where $\phi = \phi(t,x_0)$ is the transition map (independent of $u$) of system $\Sigma$ as given in Definition \ref{def:controlSystem}.
    
    This system is IOS since 
    \begin{align*}
        \abs{y(t,x_0)} = \abs{\sin(\phi(t,x_0))} \leq \abs{\phi(t,x_0)} = e^{-t}\abs{x_0}.
    \end{align*}
    However, for $x_0 = \pi$, $u \equiv 0$ it follows that $y(0,x_0) = 0$ but $y(1,x_0) = \sin(\pi e^{-1}) \neq 0$. Hence, the system is not OL.
    \hspace*{\fill}~\QED
\end{example}

    




For systems with full-state output, the notions of OL and OUGS both reduce to UGS and both OGULIM and OOULIM become ULIM. However, these notions differ for general output systems and the implication ISS$\iff$ULIM $\land$ UGS (Cor. \ref{cor:ISSsuperpositions}) cannot be extended to output systems in a naive way as stated in the following example.
In particular, OUGS together with neither OULIM nor its variations are strong enough to conclude IOS without OL as a precondition.
\begin{example}
\label{ex:OGULIMandOUGSnottoIOS}
    \emph{We show the following:
    \begin{center}
    OGULIM $\land$ OOULIM $\land$ OUGS$\nimplies$IOS $\lor$ OL.    
    \end{center}
    }
    
    We consider a two-dimensional uncontrolled system with state $x = (x_1,x_2)^T \in \R^2$ given in polar coordinates $\rho = \sqrt{x_1^2 + x_2^2} = \norm{x}_2$ and $\theta = \arg(x_1 + ix_2)$ by
    \begin{align*}
        \Sigma\colon\quad \dot \theta = 1, \quad \dot \rho = 0
        , \quad y(t,x) = \phi_1(t,x)
    \end{align*}
    with transition map (in Cartesian coordinates) $\phi(\ph,x_0) = (\phi_1(\ph,x_0),\phi_2(\ph,x_0))^T$ of $\Sigma$ corresponding to the initial condition $x_0$ represented by $(\theta_0,\rho_0)$ in polar coordinates.
    The system $\Sigma$ is OGULIM and OOULIM as it holds that
    \begin{align*}
        \phi(t,x_0) = 
        \begin{pmatrix}
            \rho_0 \cos(t + \theta_0) \\
            \rho_0 \sin(t + \theta_0)
        \end{pmatrix}
        ,
    \end{align*}
    i.e., $y(t,x_0) = 0$ for $t \in \pi\paren{\N_0 + \frac{1}{2}} - \theta_0$. 
    
    Hence, we can choose the uniform bound $\tau = \pi$ for which for any initial condition $x_0$ there exists $t \leq \tau$ such that $y(t,x_0) =0$, which implies OGULIM and OOULIM.
    
    Moreover, $\Sigma$ is OUGS as $\abs{y(t,x_0)} \leq \norm{\phi(t,x_0)}_2 = \rho_0$ $\forall t \in I$, but it is not IOS as $y(t,x_0) = \rho_0$ for  $t = 2\pi\N - \theta_0$.

    Furthermore, the system is not OL as for $x_0 = (0,1)^T$, it holds that $y(0,x_0) = 0$, but $y(\frac{3}{2}\pi,x_0)  = 1$.
    \hspace*{\fill}~\QED
\end{example}

Also, the implication ULIM $\land$ ULS $\land$ BRS$\implies$ISS or even  ULIM $\land$ ULS $\land$ BRS$\implies$UGS cannot be generalized to output systems as stated in the following example. In particular, in Lemma \ref{lem:OOULIMandLocalOLandOBORStoGlobalOL}, OOULIM cannot be exchanged by OGULIM.

\begin{example}[OGULIM $\land$ local OL $\land$ OBORS\hspace{-1pt}$\nimplies$\hspace{-1pt}OL]\label{ex:OGULIMandLocalOLandOBORSnottoGlobalOL}
    We consider the uncontrolled system $x = (x_1,x_2)^T \in \R^2$ with polar coordinates $\rho = \sqrt{x_1^2 + x_2^2} = \norm{x}_2$, $\theta = \arg(x_1 + ix_2)$ given by
	\begin{align*}
		\dot \theta &= \sat\!\paren{\tfrac{1}{\rho}}, \qquad
		\dot \rho = -\sat(\rho),
		\\
		y(t,x_0) &= \sqrt{\phi_1(t,x_0)^2 + \sat(\phi_2^2(t,x_0))},
	\end{align*}
	with transition map $\phi(\ph,x_0) = (\phi_1(\ph,x_0),\phi_2(\ph,x_0))^T$, and $\sat\colon \R_+ \to [0,1]$, $\sat(\rho) = \min\!\braces{\rho,1}$.
    First consider the following: Due to 
    \begin{align}\label{ineq:estimateRadiusOL}
		\dot \rho = -\min\!\braces{\rho,1} < 0,\qquad \rho > 0,
	\end{align}
    $\norm{\phi(t,x_0)}_2$ is strictly decreasing to zero in time and for all $\eps > 0$ and all $\rho_0 \in [0,1]$, it holds that 
    \begin{align}\label{ineq:estimateExampleOGULIM}
        \abs{y(t,x_0)} = \norm{\phi(t,x_0)}_2 = e^{-t}\norm{x_0}_2 \leq \eps
    \end{align}
    for all $t \geq \tau_1\!\paren{\eps,\rho_0} = \max\!\braces{\ln\!\paren{\frac{\rho_0} \eps},0}$ and all $x_0:\ \norm{x_0}_2 \leq  \rho_0$. Here, we used \eqref{ineq:estimateRadiusOL} and that for $\rho_0 \leq \eps$, the bound is already satisfied  at $t = 0$. For $\norm{x_0}_2 > 1$, $t = \norm{x_0}_2 -1$, it holds that $\abs{y(t,x_0)} \leq \norm{\phi(t,x_0)}_2 \leq \norm{x_0}_2-t = 1$. Hence, OGULIM follows by $\tau \coloneq \tau_1 + \max\{\norm{x_0}_2 - 1,0\}$. 
    
    For $\norm{x_0}_2 < 1$, the system is OULS by \eqref{ineq:estimateExampleOGULIM}.
    Therefore, as $\norm{x_0}_2 = y(0,x_0)$ for $\norm{x_0}_2 < 1$ and $\norm{x_0}_2 \geq 1$ implies $y(0,x_0) \geq 1$, it follows that $y(t,x_0) \leq \norm{x_0}_2^2 = y(0,x_0)$ for all $x \in B_1$, $t \in I$, i.e., the system is locally OL.
%

	Furthermore, the system is OBORS as due to
    \begin{align*}
        \dot y(t,x_0) &=
        \begin{cases}
            - \sat(\rho), &\text{if } \abs{\phi_2(t,x_0)} \leq 1,\\
            -\frac{\rho\cos^2(\theta) + \rho^2\cos(\theta)\sin(\theta)\frac{1}{\rho}}{\sqrt{\rho^2\cos^2(\theta)+ 1}}, &\text{if } \abs{\phi_2(t,x_0)} > 1,
        \end{cases}
        \\
        &\leq 0 + \tfrac{\rho \abs{\cos(\theta)}}{\sqrt{\rho^2\cos^2(\theta)+ 1}}\cdot \abs{\sin(\theta)}
        \leq 1 \cdot \abs{\sin(\theta)} 
        \leq 1,
    \end{align*}
    it holds that $y(t,x_0) \leq y(0,x_0) + t$.
	
    The system is not OL as for $\norm{x_0}_2 > 1$, $t < \norm{x_0}_2 - 1$, it holds that $\norm{\phi(t,x_0)}_2 = \norm{x_0}_2 - t$,  $\theta(t) = \theta_0 + \ln(\norm{x_0}_2)-\ln(\norm{x_0}_2 - t)$, and thus 
    for $x_0 = (0,c)$, $c > e^{\frac{\pi}{2}}$ and $t^* \coloneq \norm{x_0}_2\paren{1-e^{-\frac{\pi}{2}}}$, it holds that $y(0,x_0) = 1$, $\theta(t^*) = \frac{\pi}{2}$, and $y(t^*,x_0) = \abs{\phi_1(t^*,x_0)} = ce^{-\frac{\pi}{2}} \to \infty$ for $c \to \infty$.
    \hspace*{\fill}~\QED
\end{example}

One may wonder whether Proposition~\ref{prop:OCAGequivalences} remains valid if we omit BORS in condition \ref{cond:OCAG_1}.
The next example shows that this is in general not the case already for nonlinear infinite-dimensional systems without inputs over Hilbert spaces.
Note that for such systems, the OUAG property reduces to 0-UGATT \cite[Def. 5]{MiW18b}.

More precisely, in Example~\ref{ex:UGATTandUASnottoBRS}, we investigate a system that is FC $\land$ 0-UGATT $\land$ 0-UAS, but is still not BORS. The notion of 0-UAS \cite[Def. 5]{MiW18b} introduced here, is equivalent to local ISS for uncontrolled systems with full-state output.
For time-delay systems, an example of a system with such properties was recently presented in \cite{MaH24}.

Example~\ref{ex:UGATTandUASnottoBRS} is inspired by \cite[Ex. 2]{MiW18b}, where a system was presented which is FC $\land$ 0-GAS $\land$ 0-UAS, but at the same time is not BRS and not 0-UGATT (though the latter was not mentioned in \cite{MiW18b}).






\begin{example}[FC $\land$ 0-UGATT $\land$ 0-UAS$\nimplies$BORS]
\label{ex:UGATTandUASnottoBRS}
    We consider a control system with full-state output on a Hilbert space 
    \begin{align*}
        X &= \ell^2(\N_0,\R) = \braces{x = (x_n)_{n \in \N_0} \,\middle|\, \norm{x}_X < \infty}\!,
    \end{align*}
    where $\norm{x}_X =\sqrt{\sum_{n \in \N_0}x_n^2}$,
    the space of input values $U = \{0\}$, the space of input functions $\U = \{0\}$, and $Y = X$.


    We consider the following system inspired by \cite[Ex. 2]{MiW18b}:
    \begin{align*}
        \Sigma\colon 
        \left\{
        \begin{aligned}
            \dot x_n(s) 
            &= -x_n(s) + x_n^2(s)x_0(s) - x_n(s) \abs{x_n(s)} \\
            &\qquad - \tfrac{1}{n^2}x_n^3(s), & \hspace{-1cm} n  \in \N, \\
            \dot x_0(s) &= -x_0(s), \\ 
            y(s,x,u) &= \phi(s,x,u),
        \end{aligned}
        \right.
    \end{align*}
    where $x = (x_n)_{n \in \N_0}$ and $\phi$ is the unique maximal mild solution of $\Sigma$.
    
    It can be verified easily that the right-hand side of $\Sigma$ is Lipschitz-continuous on bounded balls. By \cite[Chap. 6, Thm. 1.4]{Paz83}, the existence of a unique maximal mild solution $\phi = \phi(s,(x_n)_{n \in \N_0},u) =(\phi_n(s,x_n,u))_{n \in \N_0}$ as defined in \cite[Chap. 6, Eq. (1.2)]{Paz83} is guaranteed. $\phi$ satisfies \ref{cond:identityProperty}--\ref{cond:cocycleProperty} by construction and therefore is a transition map. Hence, system $\Sigma$ defines a control system with outputs. 

    For the purpose of the following analysis, it is convenient to define $\Sigma_n$ as the subsystem of $\Sigma$ containing the $n$-th and 0-th component. \pb{The behavior of the subsystems is illustrated in Figure \ref{fig:infdimExample}.}

    System $\Sigma$ is forward complete and 0-UAS \cite[Def. 5]{MiW18b} with domain of attraction $\{x \in X\,|\,\forall n \in \N_0\colon \abs{x_n} \leq r < 1 \}$ by the same arguments as used in \cite[Ex. 2]{MiW18b} (the additional term $-x_n(s) \abs{x_n(s)}$ on the right-hand side of the first line component only causes faster convergence to the equilibrium). 
    
    We show that $\Sigma$ is 0-UGATT \cite[Def. 5]{MiW18b}. To this end, we show that for fixed $n \in \N$, $\Sigma_n$ with initial condition $\abs{x_n(0)}, \abs{x_0(0)} \leq r$ for any $r > 0$
    satisfies $\abs{x_n(s)}, \abs{x_0(s)} \leq \frac{1}{2}$ after finite time $s$ \pb{not depending} on $n \in \N$. It holds that
    \begin{align*}
        \abs{x_0(s)} = e^{-s} \abs{x_0(0)} \leq e^{-s}r \leq \tfrac 1 2
    \end{align*}
    for all $s \geq s^*:=\max\!\braces{\ln\!\paren{2r},0}$. For $s \geq s^*$, $|x_n(s)|$ can be bounded by the solution of the initial value problem \pb{
    \begin{align}
        \begin{aligned}
            \dot z(s) &=-z(s) - \tfrac{1}{2}z^2(s), \\
        z(s^*) &= \abs{x_n(s^*)},
        \end{aligned}
        \label{eq:ODEinPDEExample}
    \end{align}
    since for the right-hand side and $z \in \R$, it follows that
    \begin{align*}
        -z - \tfrac{1}{2}z^2\geq -z \pm z^2x_0 - z\abs{z} - \tfrac{1}{n^2}z^3.
    \end{align*} 
    Thus, for all $s \geq s^*$, it holds that $\dot z(s) \geq \frac{\diff}{\diff s}\abs{x_n(s)}$, i.e., solving \eqref{eq:ODEinPDEExample} yields}
    \begin{align*}
        \abs{x_n(s)} \leq z(s) = 2 \cdot \paren{\paren{1 + \tfrac{2}{z(s^*)}}e^{s-s^*} - 1}^{-1}, \quad s \geq s^*.
    \end{align*}


    
    As $\Sigma$ is forward complete, $z(s^*) = \abs{x_n(s^*)} < \infty$ and $\abs{x_n(s)} \leq z(s) \leq \frac{1}{2}$ for all $s \geq s^* + \ln(5)$. Therefore, the trajectory reaches the domain of attraction of 0-UAS for some $s \leq s^* + \ln(5)$. Then, due to 0-UAS and the cocycle property, it follows that $\Sigma$ is 0-UGATT.
    
    Next, we show that $\Sigma$ is not BORS. We consider $\Sigma_n$ \pb{and construct a lower bound $z$ such that $x_n(s) \geq z(s)$ for $s \geq 0$ such that $x_n(s) \leq n$ and appropriate initial conditions. Let $x_0(0) = 2e$, such that $x_0(s) = 2e \cdot e^{-s} \geq 2$ for all $s \in [0,1]$. We define $z$ by the differential equation
    \begin{align*}
        \dot z(s) = - 2 z(s) + z^2(s).
    \end{align*}
    Indeed, as 
    \begin{align*}
        - 2 z + z^2 \leq -z \pm z^2x_0(s) - z\abs{z} - \tfrac{1}{n^2}z^3.
    \end{align*}
    for all $s \in [0,1]$ and $z \in [0,n]$, $n \in \N$, it follows that $\dot z(s) \leq \dot x_n(s)$, i.e., $z$ is a lower bound for $x_n$ for appropriate initial condition $x_n(0) = z = c$. Moreover, for $c > 2$, $z$ has finite escape time.}
    We choose $c > 2$ such that \pb{$\dot z(s) = - 2 z(s) + z^2(s)$} blows up to infinity at $s = 1$.
    Hence, for $(x_n(0),x_0(0))^T= (c,2e)^T$, there exists $\tau_n \in (0,1)$, such that $x_n(\tau_n) = n$. Now, for $j \in \N$, we define the initial condition $x^j = (x_n^j)_{n \in \N_0} \in X$  for which
    \begin{align*}
        x^j_n =
        \begin{cases}
            2e, &\text{if } n=0,\\
            c, &\text{if } n=j,\\
            0, &\text{else}.\\
        \end{cases}
    \end{align*}
    It holds that 
    \begin{align*}
        \norm{y(s,x^j,u)}_Y &= \norm{\phi(s,x^j,u)}_X \geq \phi_j(s,x^j,u)
    \end{align*}
    and for $d \coloneq 2\sqrt{c^2 + 4e^2}$ it holds that
    \begin{align*}
        &\sup\nolimits_{s\geq 0,\, x \in B_d}\norm{y(s,x,u)}_Y 
        = \sup\nolimits_{s\geq 0,\, x \in B_d}\norm{\phi(s,x,u)}_X \\
        &\geq \sup\nolimits_{s\geq 0}\,(\phi_j(s,x^j,u))_j\geq x_j(\tau_j)
        \geq j 
    \end{align*}
    for any $j \in \N$. Therefore, $\Sigma$ is not BORS.
    \pb{\input{images/imageinfdimExample}}
    \hspace*{\fill}~\QED 
\end{example}


The following example treats a system with full-state output. Therefore, it also demonstrates that for the concepts of bUAG and UAG introduced in \cite[Def. 2.44]{Mir23}, it holds that bUAG$\nimplies$UAG.


\begin{example}\label{ex:OUAGnottoOGUAG} \emph{We show \\FC $\wedge$ OUAG with zero gain  $\wedge$  \pb{O}ULS $\nimplies$ OGUAG $\vee$ BORS}

    We modify system $\Sigma$ in Example \ref{ex:UGATTandUASnottoBRS} by choosing input value space $U = \R$ and $\U = \pb{\LL^\infty(I,U)}$ and the transformation of the time axis given by $s = s(t, u) \coloneq \int_0^t \frac{1}{1 + u^2(\theta)}\ \diff \theta \in \brackets{\frac{1}{1 + \norm{u}_\U^2}t, t}$, i.e.,
    \begin{align*}
        \widetilde \Sigma\colon 
        \left\{
        \begin{aligned}
            \dot x_n(t) 
            &= \tfrac{-x_n(t) + x_n^2(t)x_0(t) - x_n(t)\abs{x_n(t)} - \tfrac{1}{n^2}x_n^3(t)}{1 + u^2(t)}, \, n  \in \N, \\
            \dot x_0(t) &= -\tfrac{x_0(t)}{1 + u^2(t)}, \\
            &\hspace{-.75cm}\widetilde y(t,x,u) = \widetilde \phi(t,x,u),
        \end{aligned}
        \right.
    \end{align*}
    where $\widetilde \phi = \widetilde \phi(s,(x_n)_{n \in \N_0},u) =(\widetilde \phi_n(s,x_n,u))_{n \in \N_0}$ is the transition map of $\widetilde \Sigma$.

    Note that the transformation of the time axis causes slower time evolution for larger $\norm{u}_U$, e.g., for constant inputs $u \equiv u(0)$, it holds that $s = \frac{1}{1 + u^2(0)}t$. The relation between the flow of $\Sigma$ and $\widetilde \Sigma$ is then given by
    \begin{align*}
        \phi\!\paren{\int_0^t \frac{1}{1 + u^2(\theta)}\ \diff \theta,x,u} = \widetilde \phi(t,x,u)
    \end{align*}
    for all $t \in I$, $x \in X$ and $u \in \U$. 
    
    
    As $\Sigma$ is 0-UGATT and $t \in \brackets{s,(1 + \norm{u}_\U^2)s}$, for every $C > 0$ and $u \in B_{C,\U}$, $\widetilde \Sigma$ satisfies OUAG with zero gain.
		   


		
  
        
    However, $\widetilde \Sigma$ is not OGUAG: For all $j \in \N$, we consider $x^j \in X$, $\tau_j \in I$ and $c, d > 0$ as defined in Example \ref{ex:UGATTandUASnottoBRS}.
    
    $(\phi(s,x^j,u))_j$ is smaller than the solution of $\dot z =2ez^2$ with initial condition $z(0) = c$, i.e., $(\phi(s,x^j,u))_j \leq z(s) =\frac{c}{1 - 2ecs}$ for $s < \frac{1}{2ec}$. Especially, $(\phi(s,x^j,u))_j \leq 2c$ for $s \leq \frac{1}{4ec}$, which means that $\tau_j \in \left[\frac{1}{4ec},1\right)$ for sufficiently large $j \in \N$.
    
    By the transformation of the time axis, we obtain for any $\tau \geq 1$ and $u^j$ defined by $u^j \equiv \sqrt{\frac{\tau}{\tau_j} - 1}$ that $t = (1 + \|u^j\|^2_{\U}) s = \frac{\tau}{\tau_j}s$. Then, for $x^j \in B_d$, $u^j \in B_{2\sqrt{4ec \tau - 1},\U}$, it holds that
    \begin{align*}
        \norm{\widetilde y(\tau,x^j,u^j)}_Y 
        &=\|\widetilde \phi(\tau,x^j,u^j)\|_X \\
        &\geq \big(\widetilde \phi\!\paren{(1 + \|u^j\|^2_{\U})\tau_j,x^j,u^j}\big)_j \\
        &= (\phi\!\paren{\tau_j,x^j,u^j})_j
        \geq j \to \infty \text{ for } j \to \infty.
    \end{align*}
    Hence, for every $\gamma \in \KK_\infty$, some $\eps > 0$, $r \coloneq d$ and every $\tau = \tau(\eps,r) \geq 1$, there exist $x \in B_r$ and $u \in \U$ such that 
    \begin{align*}
        \norm{\widetilde y(\tau,x,u)}_Y &> \eps + \gamma(\norm u_\infty).
    \end{align*}
    This means, that $\widetilde \Sigma$ is not OGUAG.
    \hspace*{\fill}~\QED 
\end{example}

    
    
    


\section{Conclusion}\label{sec:conclusion}
The main results of this work are superposition theorems for IOS, IOS $\land$ OL and OCAG for infinite-dimensional systems \pb{(see Fig. \ref{fig:mainResult})}. Thereby, we set the basis for developing the infinite-dimensional IOS theory.
On this path, it is necessary to introduce several stability and attractivity notions and to systematically analyze the relations between them.

We prove that our results generalize the existing theory for  ODE systems \cite{ISW01}. However, by means of counterexamples, we show that not all of the characterizations for ODEs hold in general in the infinite-dimensional case, e.g., OUAG$\nimplies$OGUAG$\nimplies$IOS. 

Our results generalize the ISS superposition theorems from \cite{MiW18b} to systems with outputs. In our setting, several stability notions appear that reduce to the same notions for systems with full-state output, e.g., OGULIM and OOULIM both reduce to ULIM in the case of full-state output. Even more, the notions OGULIM and OOULIM together combined with OUGS are not sufficient to conclude IOS (see Example \ref{ex:OGULIMandOUGSnottoIOS}), opposed to ULIM $\land$ UGS$\iff$ISS which holds for systems with full-state output \cite[Thm. 5]{MiW18b}.

We show that OOULIM, local OL and OBORS together imply OL, and we characterized ISS by IOS and IOSS.

In our future work, we aim at further developing the IOS theory for infinite-dimensional systems by providing Lyapunov characterizations and small gain theorems for the analysis of interconnected systems. \pb{We will illustrate its applicability by practical examples.}

\section*{References}
\bibliographystyle{IEEEtran}
\bibliography{Bibliothek}

\begin{IEEEbiography}[{\includegraphics[width=1in,height=1.25in,clip,keepaspectratio]{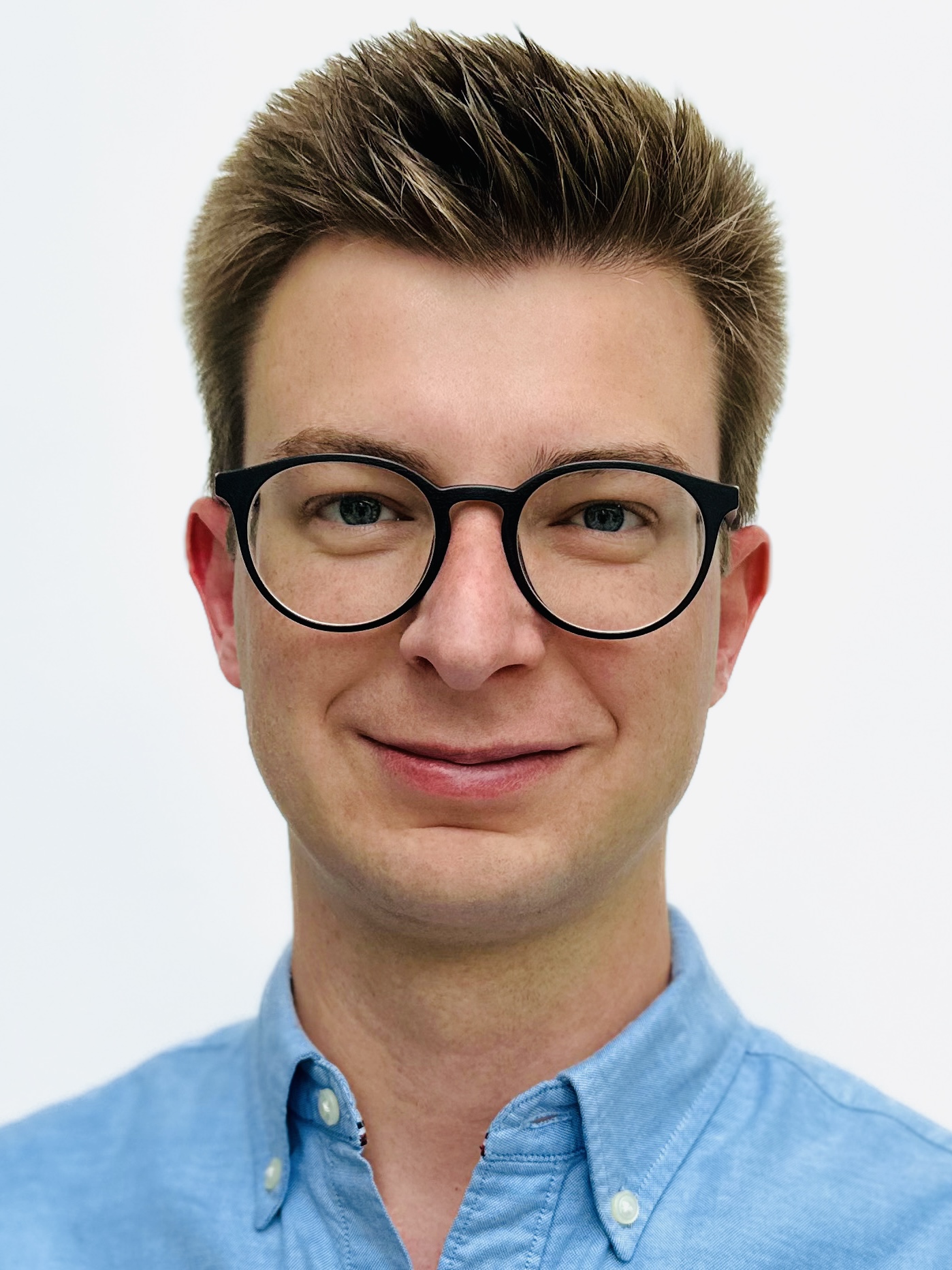}}]{Patrick Bachmann} was born in Speyer, Germany, in 1992. He received his Bachelor's degree in business mathematics from the University of Mannheim, Germany, in 2015 and his Master's degree in Mathematics from Karlsruhe Institute of Technology in 2018. 
\pb{He worked as a research assistant at the Technical University of Kaiserslautern, Germany, and the University of Würzburg, Germany. Currently, he is working at the University of Bayreuth, Germany, while pursuing his PhD degree in Mathematics under supervision of Sergey Dashkovskiy and Andrii Mironchenko.} His research interests include impulsive systems, stability and control theory, Lyapunov functions and infinite-dimensional systems.
\end{IEEEbiography}

\begin{IEEEbiography}[{\includegraphics[width=1in,height=1.25in,clip,keepaspectratio]{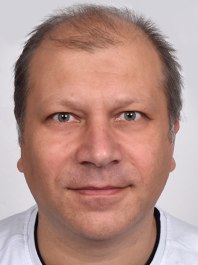}}]{Sergey Dashkovskiy} received the M.Sc.
degree in applied mathematics from the
Lomonosov University of Moscow, Russia, in 1996, 
the Ph.D. degree in mathematics from the University of Jena, Germany,
in 2002, and the Habilitation (venia legendi) degree in mathematics 
from the University of Bremen, Bremen, Germany, in 2009.
He held positions with the Arizona State
University, Tempe, AZ, USA, the University of
Bayreuth, Bayreuth, Germany, and the University of Applied Sciences Erfurt, Germany. 
Since 2016, he has been a Professor and the Head of the Research Group Dynamics and
Control, Institute for Mathematics, University of Würzburg, Germany. 
His research interests are in stability theory of dynamical
systems and networks.
Dr. Dashkovskiy is currently the Editorial Board Member of several
journals related to this research area, in particular, of IEEE Transactions on Automatic Control 
and IET Journal on Control Theory and Applications.
\end{IEEEbiography}

\begin{IEEEbiography}[{\includegraphics[width=1in,height=1.25in,clip,keepaspectratio]{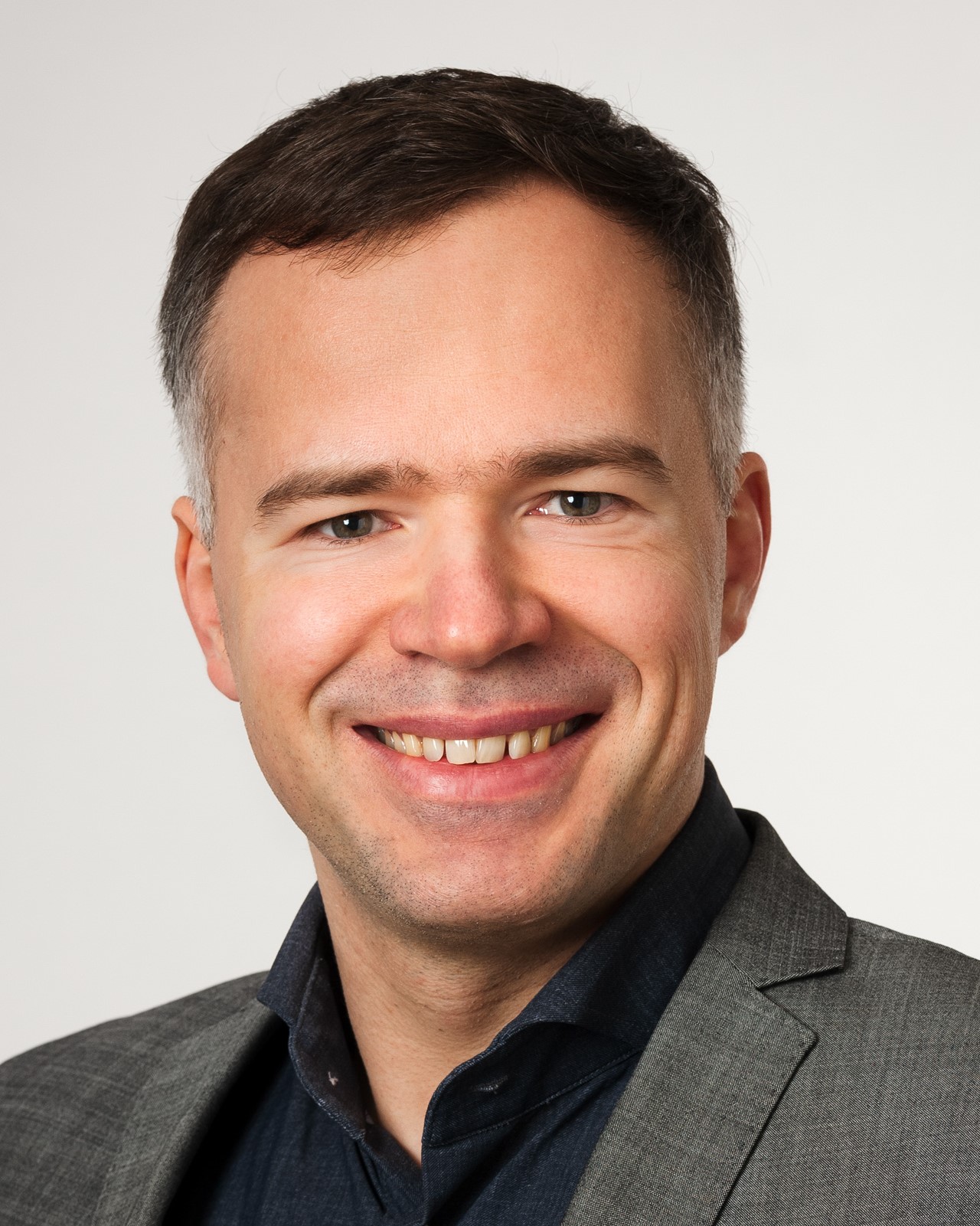}}]{Andrii Mironchenko} (M'21-SM'22) 
was born in 1986 in Odesa, Ukraine. \amc{He received his M.Sc. degree from Mechnikov Odesa National University, Ukraine; his} Ph.D. degree in mathematics from the University of Bremen, Germany in 2012, and a Habilitation degree from the University of Passau, Germany (honored by 2024 Outstanding Habilitation Award of the University of Passau). He was a Postdoctoral Fellow of the Japan Society for Promotion of Science (2013--2014). 
\amc{Since December 2024, he has been with the Department of Mathematics, University of Bayreuth, Germany.}
Dr. Mironchenko is the author of the monograph „Input-to-State Stability: Theory and Applications“ (Springer, 2023) and of over 70 journal and conference papers on control theory and applied mathematics. 
A. Mironchenko is an Associate Editor in Systems \& Control Letters and is a co-founder and co-organizer of the biennial Workshop series “Stability and Control of Infinite-dimensional Systems” (SCINDIS, 2016 -- now). He received the 2023 IEEE CSS George S. Axelby Outstanding Paper Award and a Heisenberg grant from German Research Foundation (2024).
He works in stability and control of nonlinear systems, distributed parameter systems, hybrid systems, and applications of control theory.
\end{IEEEbiography}

\end{document}

%% file: images/implicationDiagram.tex
\begin{figure*}[htbp]
    \vspace{.3cm}
    \centering
    \begin{tikzpicture}
        \node (IOS) at (-6,0) {IOS};
        \node[right=.8cm of IOS.east] (OUAG+OCEP+BORS) {OUAG $\land$ OCEP $\land$ BORS};
        \node[right=.8cm of OUAG+OCEP+BORS.east] (N5) {OUAG $\land$ OULS $\land$ BORS};
        \node[right=.8cm of N5.east] (N6) {OUAG $\land$ OUGS};
        \node[right=.8cm of N6.east] (N7) {OCAG $\land$ OULS};
        \node[fit=(IOS) (N7),draw=none,inner ysep=0mm](dummyBlockIOS) {};
        \node[above=0.1cm of dummyBlockIOS,anchor=base] (titleOLIOS) {\footnotesize Theorem \ref{thm:IOSequivalences}};
        \node[rectangle,fit=(IOS) (N7) (titleOLIOS),draw=black,inner ysep=.5mm](BlockIOS) {};

        \node[above = 1.2cm of BlockIOS] (dummyOLIOS) {};
        \node[left = -0.8cm of dummyOLIOS.center] (OUAG+OL) {OUAG $\land$ OL $\land$ $h$ bounded};
        \node[left = 1cm of OUAG+OL] (IOS+OL) {IOS $\land$ OL};
        \node[right = 1cm of OUAG+OL] (OULIM+OL) {OULIM $\land$ OL $\land$ $h$ bounded};
        \node[fit=(IOS+OL)(OULIM+OL),draw=none,inner ysep=0mm](dummyBlockOLIOS) {};
        \node[above=0.1cm of dummyBlockOLIOS,anchor=base] (titleOLIOS) {\footnotesize Proposition \ref{prop:OLIOSequivalences}};
        \node[rectangle,fit=(IOS+OL)(OULIM+OL)(titleOLIOS), draw=black,inner ysep=.5mm](BlockIOSOL) {};
        \path
        (IOS+OL) edge[thick,double,double equal sign distance,{Implies[]}-{Implies[]}] (OUAG+OL)
        (OULIM+OL) edge[thick,double,double equal sign distance,{Implies[]}-{Implies[]}] (OUAG+OL);

        \node[below = 0.9cm of BlockIOS] (OULIM+OUGS) {OULIM $\land$ OUGS};
        %
        \path
        ([shift={(-0.2cm,-0.3cm)}]dummyOLIOS.south) edge[thick,double,double equal sign distance,-{Implies[]}]
        ([shift={(-0.2cm,0.1cm)}]BlockIOS.north)
        ([shift={(0.2cm,0.1cm)}]BlockIOS.north) edge[thick,double,double equal sign distance,-{Implies[]},degil]
        node[right=.2cm]{\footnotesize Example \ref{ex:IOSnottoOL}} ([shift={(0.2cm,-0.3cm)}]dummyOLIOS.south)
        (IOS) edge[thick,double,double equal sign distance,{Implies[]}-{Implies[]}] (OUAG+OCEP+BORS)
        (OUAG+OCEP+BORS) edge[thick,double,double equal sign distance,{Implies[]}-{Implies[]}] (N5)
        (N5) edge[thick,double,double equal sign distance,{Implies[]}-{Implies[]}] (N6)
        (N6) edge[thick,double,double equal sign distance,{Implies[]}-{Implies[]}] (N7);
        %
        \path([shift={(-0.2cm,-0.1cm)}]BlockIOS.south) edge[thick,double,double equal sign distance,-{Implies[]}]
        node[left=.2cm]{\footnotesize Lemma \ref{lem:IOStoBORSandOCEPandOGUAGandOGULIM}} ([shift={(-0.2cm,0cm)}]OULIM+OUGS.north)
        ([shift={(0.2cm,0cm)}]OULIM+OUGS.north) edge[thick,double,double equal sign distance,-{Implies[]},degil]
        node[right=.2cm]{\footnotesize Example \ref{ex:OGULIMandOUGSnottoIOS}} ([shift={(0.2cm,-0.1cm)}]BlockIOS.south);
    \end{tikzpicture}
    \caption{Diagram of implications.}
    \label{fig:mainResult}
\end{figure*}

%% file: images/implicationDiagramIOS_new.tex
\begin{figure}[htb]
    \centering
    \begin{tikzpicture}
        \node (IOS) at (0,0) {IOS};
        \node (OUAG+OCEP+BORS) at (4.4,-1.1) {OUAG $\land$ OCEP $\land$ BORS};
        \node (OUAG+OULS+BORS) at (0,-2.2) {OUAG $\land$ OULS $\land$ BORS};
        \node (OUAG+OUGS) at (4.4,0) {OUAG $\land$ OUGS};
        \node (OCAG+OULS) at (0,-1.1) {OCAG $\land$ OULS};
        \node (OCAG+OCEP) at (4.4,-2.2) {OCAG $\land$ OCEP};
        %
        \path
        (IOS) edge[thick,double,double equal sign distance,-{Implies[]}] node[above]{\footnotesize Lem. \ref{lem:IOStoBORSandOCEPandOGUAGandOGULIM}} (OUAG+OUGS)
        (OUAG+OUGS) edge[thick,double,double equal sign distance,-{Implies[]}] node[right=2mm]{\footnotesize Lem. \ref{lem:IOStoBORSandOCEPandOGUAGandOGULIM}} (OUAG+OCEP+BORS)
        (OUAG+OCEP+BORS) edge[thick,double,double equal sign distance,-{Implies[]}] node[below right=-0mm and -2mm]{\footnotesize Lem. \ref{lem:OUAGandOCEPtoOULS}} (OUAG+OULS+BORS)
        (OUAG+OULS+BORS) edge[thick,double,double equal sign distance,-{Implies[]}] node[left=2mm]{\footnotesize Prop. \ref{prop:OCAGequivalences}} (OCAG+OULS)
        (OCAG+OULS) edge[thick,double,double equal sign distance,-{Implies[]},align=center] node[left=2mm]{{\footnotesize Prop. \ref{prop:OCAGequivalences},}\\ {\footnotesize Lem. \ref{lem:OUGBandOULStoOUGS}, \ref{lem:OCAGandOUGStoIOS}}} (IOS)
        (OUAG+OCEP+BORS) edge[thick,double,double equal sign distance,{Implies[]}-{Implies[]}] node[right=2mm]{\footnotesize Prop. \ref{prop:OCAGequivalences}} (OCAG+OCEP);
    \end{tikzpicture}
    \caption{Diagram of implications for the proof of Theorem \ref{thm:IOSequivalences}.}
    \label{fig:implicationDiagramIOS}
\end{figure}

%% file: images/IOSimplications.tex
\begin{figure}[htbp]
    \centering
    \begin{tikzpicture}
        \node (IOS) at (0,0) {IOS};

        \node[below left =1.0cm and 2.2cm of IOS.mid,anchor=mid] (OUGS) {OUGS};
        \node[below right =1.0cm and 2.2cm of IOS.mid,anchor=mid] (OCAG) {OCAG};
        \node[below =1.0cm of OCAG.mid,anchor=mid] (OGUAG) {OGUAG};
        \path
        (IOS) edge[thick,double,double equal sign distance,-{Implies[]}] node[anchor=south east]{} (OUGS)
        (IOS) edge[thick,double,double equal sign distance,-{Implies[]}] node[anchor=south west]{} (OCAG)
        (OCAG) edge[thick,double,double equal sign distance,-{Implies[]}] node[anchor=south west]{} (OGUAG);

        \node[below left =1.0cm and 1.0cm of OUGS.mid,anchor=mid] (OULS) {OULS};
        \node[below right =1.0cm and 1.0cm of OUGS.mid,anchor=mid] (OUGB) {OUGB};
        \path
        (OUGS) edge[thick,double,double equal sign distance,-{Implies[]}] node[anchor=south east]{} (OUGB)
        (OUGS) edge[thick,double,double equal sign distance,-{Implies[]}] node[anchor=south west]{} (OULS)
        (OCAG) edge[thick,double,double equal sign distance,-{Implies[]}] node[anchor=south east]{} (OUGB);

        \node[below left =1.0cm and 1.0cm of OGUAG.mid,anchor=mid] (OUAG) {OUAG};
        \node[below right =1.0cm and 1.0cm of OGUAG.mid,anchor=mid] (OGULIM) {OGULIM};
        \path
        (OGUAG) edge[thick,double,double equal sign distance,-{Implies[]}] node[anchor=south east]{} (OUAG)
        (OGUAG) edge[thick,double,double equal sign distance,-{Implies[]}] node[anchor=south west]{} (OGULIM);

        \node[below =1.0cm of OUGB.mid,anchor=mid] (BORS) {BORS};
        \node[below = (1.3cm) of BORS.mid,align=center,anchor=mid] (hbounded) {$h$ is bounded\\on bounded sets};
        \path
        (OUGB) edge[thick,double,double equal sign distance,-{Implies[]}]  node[anchor=east]{} (BORS)
        (BORS) edge[thick,double,double equal sign distance,-{Implies[]}]  node[anchor=east]{} (hbounded);
        
        \node[below =1.0cm of OULS.mid,anchor=mid] (OCEP) {OCEP};
        \path
        (OULS) edge[thick,double,double equal sign distance,-{Implies[]}]  node[anchor=east]{} (OCEP);


        \node[below =1.0cm of OGULIM.mid,anchor=mid] (OULIM) {OULIM};
        \path
        (OUAG) edge[thick,double,double equal sign distance,-{Implies[]}] node[anchor=south west]{} (OULIM)
        (OGULIM) edge[thick,double,double equal sign distance,-{Implies[]}] node[anchor=south east]{} (OULIM);
        \node[below =1.0cm of OUAG.mid,anchor=mid] (OAG) {OAG};
        \node[below =1.0cm of OULIM.mid,anchor=mid] (OLIM) {OLIM};
        \path
        (OUAG) edge[thick,double,double equal sign distance,-{Implies[]}] node[anchor=south east]{} (OAG)
        (OAG) edge[thick,double,double equal sign distance,-{Implies[]}] node[anchor=south west]{} (OLIM)
        (OULIM) edge[thick,double,double equal sign distance,-{Implies[]}] node[anchor=south east]{} (OLIM);
    \end{tikzpicture}
    \caption{Diagram of elementary implications summarized in Lemma \ref{lem:IOStoBORSandOCEPandOGUAGandOGULIM}.}
    \label{fig:IOSimplications}
\end{figure}

%% file: images/implicationDiagramOCAG.tex
\begin{figure}[htb]
    \centering
    \begin{tikzpicture}
        \node (OUAG) at (0,0) {OUAG $\land$ BORS};
        \node (OGUAG) at (0,1) {OGUAG $\land$ OUGB};
        \node (OCAG) at (3,0.5) {OCAG};
        \node (IOpS) at (5,0.5) {IOpS};
        
        \path
        (OUAG) edge[thick,double,double equal sign distance,-{Implies[]}] node[anchor=south]{} (OGUAG)
        (OGUAG) edge[thick,double,double equal sign distance,-{Implies[]}] node[right=.2cm]{} (OCAG)
        (OCAG) edge[thick,double,double equal sign distance,-{Implies[]}] node[anchor=north]{} (OUAG)
        (OCAG) edge[thick,double,double equal sign distance,-{Implies[]}] node[left=.2cm]{} (IOpS);
    \end{tikzpicture}
    \caption{Diagram of implications for the proof of Theorem \ref{prop:OCAGequivalences}.}
    \label{fig:implicationDiagramOCAG}
\end{figure}

%% file: images/imageinfdimExample.tex
\begin{figure}[h]
    \centering
    \begin{tikzpicture}
    \pgfplotsset{
    y coord trafo/.code={
        \pgfmathparse{
            #1>5?
                #1+70
            :
                #1*15
        }
    },
}
        \begin{axis}[
            width=\columnwidth,
            height=0.5\columnwidth,
            axis lines=center,
            xlabel style={below},
            ylabel style={above},
            xtick={1,6},
            xticklabels={1,$s^*$},
            ytick=\empty,
            xlabel={$t$},
            xmax=5,
            xmin=0,
            ymax=65,
            ymin=0,
            extra y ticks={0,1,2.3130,5,20,40,60},
            extra y tick style={
                tick label style={
                    font = \scriptsize,
                    anchor= east,
            }},
            extra y tick labels={0,1,$c$,5,20,40,60},
            smooth
            ]
            \addplot[color=uniwueblue,
            unbounded coords=jump
            ] table{images/rawdata/plot_1.txt};
            \addplot[color=uniwueblue!80,
            unbounded coords=jump
            ] table{images/rawdata/plot_2.txt};
            \addplot[color=uniwueblue!60,
            unbounded coords=jump
            ] table{images/rawdata/plot_3.txt};
            \addplot[color=uniwueblue!40,
            unbounded coords=jump
            ] table{images/rawdata/plot_4.txt};
            \addplot[color=black,
            unbounded coords=jump
            ] table{images/rawdata/plot_lowerbound.txt};
            \addplot[color=red,
            unbounded coords=jump
            ] table{images/rawdata/plot_upperbound.txt};
            \addplot[soldot] coordinates{(0,2.3130)};
            \draw[dashed,black] (1,0) to (1,50);
            \draw[double,black] (0,5) to (10,5);
            \node at (0.8,1) {$\phi_1$};
            \node at (1.2,2.4) {$\phi_2$};
            \node at (1.4,3.6) {$\phi_3$};
            \node at (1.9,4.2) {$\phi_4$};
        \end{axis}
    \end{tikzpicture}
    \caption{\pb{Plot for several components $\phi_n$ with initial condition $x_n = c$. Local lower bound for $\phi(n) \leq n$ in black and upper bound for global attractivity in red.}}
    \label{fig:infdimExample}
\end{figure}